\documentclass[12pt,reqno]{amsart}
\usepackage[margin=1in]{geometry}

\newcommand{\DateOfPub}[1]{}

\usepackage[UKenglish]{babel}

\usepackage{amsfonts,amsmath,amsthm,amssymb} % \text and \mathbb command
\usepackage[foot]{amsaddr}
\usepackage{mathrsfs} % script type font
\usepackage{stmaryrd} % new bracket
\usepackage{mathtools}

\usepackage{pifont} % for cmark and xmark
\newcommand{\cmark}{\ding{51}}%
\newcommand{\xmark}{\ding{55}}%

\usepackage{lscape} %to make the page landscape

\usepackage{natbib}
\usepackage{url}

\usepackage{comment} % commenting out multiple lines
\usepackage{graphicx} % graphics
\usepackage{setspace}
\usepackage{color}

\usepackage{multirow}

\allowdisplaybreaks[1]

\newtheorem{theorem}{Theorem}[section]%[subsection]
\newtheorem{proposition}[theorem]{Proposition}
\newtheorem{lemma}[theorem]{Lemma}
\newtheorem{corollary}[theorem]{Corollary}

\theoremstyle{remark}
\newtheorem{remark}{Remark}
\numberwithin{equation}{section}

\newcommand\EatDot[1]{}

\newcommand{\E}{\mathbf{E}}
\newcommand{\F}{\mathbf{F}}
\newcommand{\N}{\mathbf{N}}
\newcommand{\R}{\mathbf{R}}

\newcommand{\diff}{\mathrm{d}}
\newcommand{\resid}{\mathcal{R}}

\newcommand{\sfG}{\mathsf{G}}
\newcommand{\sfK}{\mathsf{K}}
\newcommand{\calF}{\mathcal{F}}

\newcommand{\TV}{\mathrm{TV}}
\newcommand{\vol}{\mathrm{vol}}
\newcommand{\zero}{\mathbf{0}}
\newcommand{\generator}{\mathcal{L}}
\newcommand{\AModel}{A^{\mathrm{m}}}
\newcommand{\bModel}{b^{\mathrm{m}}}
\newcommand{\Proj}{\mathrm{Proj}}

\DeclareMathOperator{\tr}{tr}

\DeclareMathOperator*{\argmin}{argmin}

\title[Estimation of SDE via OGD]{Parametric estimation of stochastic differential equations via online gradient descent}
\author[S Nakakita]{Shogo Nakakita}
\address{Komaba Institute for Science\\
the University of Tokyo\\
3-8-1 Komaba, Meguro-ku, Tokyo 153-8902, Japan}
\begin{document}
\maketitle

\begin{abstract}
    We propose an online parametric estimation method of stochastic differential equations with discrete observations and misspecified modelling based on online gradient descent.
    Our study provides uniform upper bounds for the risks of the estimators over a family of stochastic differential equations.
    The derivation of the bounds involves three underlying theoretical results: the analysis of the stochastic mirror descent algorithm based on dependent and biased subgradients, the simultaneous exponential ergodicity of classes of diffusion processes, and the proposal of loss functions whose approximated stochastic subgradients are dependent only on the known model and observations.
    
\end{abstract}

\section{Introduction}

\subsection{Problem}
Let us consider the parametric estimation of the following $d$-dimensional stochastic differential equation (SDE):
% Let us consider parametric estimation of the drift coefficient $b$ of the following $d$-dimensional stochastic differential equation
\begin{align*}
    \diff X_{t}^{a,b}=b\left(X_{t}^{a,b}\right)\diff t+a\left(X_{t}^{a,b}\right)\diff w_{t},\ X_{0}^{a,b}=x\in\R^{d},\ t\ge0.
\end{align*}
% which defines a $d$-dimensional diffusion process $X_{t}^{a,b}$, $t\ge0$.
SDEs describe dynamics with randomness and allow for flexible model structures under mild conditions.
Therefore, they are used to model phenomena in broad disciplines such as finance, biology, epidemiology, physics, meteorology, and machine learning.
In this study, we propose an online parametric estimation method of $a$ and $b$ based on $\{X_{ih_{n}}^{a,b}\}_{i=0,\ldots,n}$ with $h_{n}>0$, discrete observations of the diffusion process $\{X_{t}^{a,b}\}_{t\ge0}$ defined by the SDE.% without linearity of the model in its parameters.

Batch estimation of SDEs with discrete observations is a classical and important problem for statistics of SDEs \citep[see][]{Flo89,Yos92,GJ93,BS95,Kes97,Hof99,KS99,GHR04}.
In recent studies, this topic has been intensively investigated from various statistical perspectives such as asymptotic expansions of the densities for maximum likelihood estimation \citep{Li13,Choi15}, nonparametric estimation \citep{GL14,CoG21}, nonparametric Bayesian estimation \citep{vdMvZ13,NS17}, and quasi-likelihood analysis \citep{Yos11,Mas13,UY13}.
Moreover, diverse model settings such as L\'{e}vy-driven SDEs \citep{SY06,Mas13,GLM18,ClG20,MMU22}, small diffusions \citep{SU03,GS09,GLV14}, and stochastic partial differential equations \citep{BT20,Cho20,HT21} have been examined.
% Moreover, the estimation under diverse model settings such as L\'{e}vy-driven SDEs \citep{SY06,Mas13,GLM18,ClG20,MMU22}, small diffusions \citep{SU03,GS09,GLV14}, and stochastic partial differential equations \citep{BT20,Cho20,HT21} has been examined.
Notably, some studies propose computationally efficient batch estimators and show that they achieve asymptotic efficiency \citep{UY12,KUY17,KU18a,KU18b}.

Online estimation, where the estimator is updated as data are acquired, is also a typical and significant concern in time series data analysis because it is quite useful for real-time decision making.
For example, Kalman filtering is one of the most classical online estimation methods based on time series data.
However, most studies on the online parametric estimation of SDEs depend on the setting of continuous observations $\{X_{t}^{a,b}\}_{t\ge0}$ \citep{SP19,BC21,SK22}, which is restrictive in real data analysis.
% However, most studies on the online parametric estimation of SDEs depend on settings that are often too restrictive; for example, the existence of explicit representation of transition density functions \citep{PSD06} or the setting of continuous observations $\{X_{t}^{a,b}\}_{t\ge0}$ \citep{SP19,BC21,SK22}.
Hence, we aim to propose online estimation methods for SDEs with discrete observations.%  without explicit density functions.

We provide uniform upper bounds for the risk of the parametric estimation of both diffusion and drift coefficients of SDEs with discrete observations and model misspecification via online gradient descent. % with convex loss functions and their convex approximations.
Those bounds give theoretical convergence guarantees of the proposed online estimation method for SDEs with discrete observations, which are the main contribution of our study.
To derive the bounds, we combine the three theoretical discussions: (i) model-wise non-asymptotic risk bound for the stochastic mirror descent (SMD) with dependent and biased subgradients; (ii) simultaneous ergodicity and uniform moment bounds for a class of SDEs; and (iii) the proposal of loss functions for the online parametric estimation.
%First,  we illustrate a model-wise non-asymptotic risk bound for the stochastic mirror descent (SMD) with dependent and biased subgradients.
% This result is sufficient to provide a risk bound for the parametric estimation via online gradient descent for fixed SDEs.
% Furthermore, we examine the simultaneous ergodicity and uniform moment bounds for a class of SDEs and observe that the risk bound is uniform for a particular class of SDEs.
% We combine these discussions and demonstrate that the proposed method not only achieves uniform risk bounds for fixed sample sizes and efficient computational complexities via online gradient descent but also allows for model misspecification.
Consequently, this study also contributes to technical topics such as (1) gradient-based online optimization for dependent data and (2) the simultaneous ergodicity of a family of SDEs and statistical topics such as SDEs' (3) estimation with non-asymptotic uniform risk bounds, (4) computationally efficient estimation, and (5) misspecified modelling.
% We explain the contributions in Section 1.3.

\subsection{Literature review}
% imitating SK22

\subsubsection{Gradient-based online optimization algorithms for dependent data}
Gradient-based online optimization algorithms and their asymptotic properties have been widely discussed \citep[for example, see][and the studies cited therein]{KY03}.
Moreover, recent studies consider the non-asymptotic properties of stochastic optimization algorithms under dependence.
\citet{DAJ+12} analyse SMD with mixing noises and demonstrate its non-asymptotic upper bound in expectation and with high probability.
\citet{SSY18} consider stochastic gradient descent with noises that are a Markov chain under both convex and nonconvex loss functions.
\citet{BJN+20} consider stochastic regression problems with Markov chains and propose a stochastic gradient descent algorithm based on the experience replay technique popular in reinforcement learning.
\citet{KNJ+21} examine the parametric estimation of Markov chains via stochastic gradient descent with modifications inspired by the experience replay technique.
\citet{BCM+21} investigate the 2-Wasserstein distance between the law of stochastic gradient Langevin dynamics with dependent data and the target distribution.

The recursive estimation of time series models, such as recursive maximum likelihood estimation (RMLE), is one of the classical topics in time series analysis, and some results are based on gradient-based approaches \citep[for an overview, see][]{SGH+19}.
\citet{Ryd97} discusses hidden Markov models with discrete state space and shows the gradient-based stochastic algorithm that converges to a stationary point of the Kullback--Leibler divergence.
\citet{SGH+19} show that the proposed gradient-based RMLE for hidden Markov models converges to the local minimum of the Kullback--Leibler divergence between true and estimated densities even with model misspecification.
\citet{SP19} consider stochastic gradient approaches for the RMLE of diffusion processes based on continuous but partial observations.
\citet{BC21} illustrate the convergence of online drift estimation of jump-diffusion processes developed on continuous observations and stochastic gradient descent.
\citet{SK22} investigate infinite-dimensional linear diffusion processes with partial observations and propose RMLE based on two-timescale stochastic gradient descent and stochastic filtering for the hidden state via their estimation method.

In addition to stochastic approaches, regret analyses for the parameter estimation of time series models via online optimization algorithms have also been discussed.
\citet{AHM+13} suggest online estimation methods of autoregressive moving-average (ARMA) processes based on online gradient descent and online Newton steps \citep[see][]{HAK07} and demonstrate the bounds for their regrets; \citet{LHZ+16} extend the discussion to autoregressive integrated moving-average (ARIMA) processes.

\subsubsection{Simultaneous ergodicity of a family of SDEs}
The ergodicity of solutions of SDEs, including L\'{e}vy-driven ones, is a central topic in the research on stochastic processes \citep[e.g.,][]{Ver88,Ver97,PV01,Mas07,Kul17}.
In statistical inference, the simultaneous exponential ergodicity for a family of SDEs, that is, whose convergence in total variation is uniform in the family, is useful to provide non-asymptotic uniform risk bounds.
\citet{GP14} examine the exponential ergodicity for a class of Markov chains and show the simultaneous ergodicity for a class of one-dimensional diffusion processes as applications.
\citet{Kul09} also considers the simultaneous exponential ergodicity for a class of SDEs driven by jump processes.

\subsubsection{Sequential estimation for diffusion processes}
A sequential estimation scheme proposes pairs of stopping rules and estimators for prescribed precision.
Remarkably, it provides non-asymptotic bounds on the risk of the estimators of stochastic processes \citep[see][]{LS01a,LS01b}.

% parametric one: check the proofs and see the differences
Most studies on the parametric sequential estimation of SDEs are based on the linearity of the model with respect to the parameters.
\citet{Nov72} considers sequential plans for diffusion processes and examines one-dimensional parameters.
\citet{GK97} study the plans for linear stochastic regression models being diffusion processes, whereas \citet{GK01} extend it to semimartingales.
The sequential parametric estimation of SDEs with time delay has been studied as well \citep[see][]{KV01,KV05,KV10}.

Nonparametric sequential estimation is also a topic that has gathered research interest.
\citet{GP05} propose a nonparametric sequential estimation method for diffusion processes and demonstrate that it achieves asymptotic optimality in the context of pointwise risks.

Some studies consider truncated sequential estimation, in which the stopping rules are bounded by a fixed sample size, and upper bounds for the estimators' precision are dependent on the sample size.
For instance, \citet{Vas14} studies parametric and nonparametric truncated sequential estimation based on ratio-type functionals of stochastic processes including L\'{e}vy-driven Ornstein--Uhlenbeck processes with discrete observations.
\citet{GP06,GP11,GP15} propose the nonparametric truncated sequential estimation of discretely observed diffusion processes and present non-asymptotic upper bounds for some risks in addition to asymptotic efficiency of the estimator.
\citet{GP22} discuss diffusion processes with the parameter whose dimension is greater than the sample size and obtain upper bounds for the quadratic risk of their nonparametric sequential estimators and asymptotic efficiency.

\subsubsection{Computationally efficient estimation of diffusion processes}
Several studies propose estimation methods associated with convex optimization, such as Lasso estimation or the Dantzig selector for diffusion processes with discrete observations.
\cite{DI12} demonstrate the oracle properties of adaptive Lasso-type estimation for diffusion processes with discrete observations.
\citet{MS17} investigate regularized estimation with general quasi-log-likelihood functions and its asymptotic properties, which apply to the estimation of diffusion processes.
The proposed estimators in these studies can be obtained efficiently with convex optimization algorithms if negative quasi-log-likelihood functions are convex.
\cite{Fuj19} studies the Dantzig selector for high-dimensional parameters of linear SDEs that can be solved through linear programming.

Notably, \citet{KUY17} and \citet{KU18a,KU18b} propose a hybrid estimation wherein Bayes estimators with reduced sample sizes are used as the initial values for multi-step estimators via Newton--Raphson algorithms.
It shows asymptotic efficiency and experimentally good performance for nonlinear structures of parameters.

\subsubsection{Misspecified modelling of SDEs}
Model misspecification is one of the classical topics in statistics of SDEs.
\citet{McK84}, \citet{Yos90}, and \citet{Kut04,Kut17} consider asymptotic properties of drift estimation with continuous observations and misspecified models.
\citet{UY11} study Wiener-driven SDEs with discrete observations and find that the convergence rate of diffusion estimation with misspecification can be different from those with the correct specifications.
\citet{Ueh19} illustrates that the convergence rates of parametric estimation of some L\'{e}vy-driven SDEs remain the same regardless of model misspecification.
\citet{Ogi21} exhibits the asymptotic mixed normality of an estimator of diffusion coefficients of Wiener-driven SDEs under model misspecification, fixed observation terminals, and discrete and noisy observations and applies neural network modelling in estimation.

\subsection{Contributions}
Our result on uniform upper bounds for the risks of the online parametric estimation is based on three technical contributions that are of independent interest: (i) model-wise non-asymptotic upper bounds for the risk of SMD with dependent and biased subgradients; (ii) the simultaneous exponential ergodicity for a class of SDE models; (iii) the proposal of convex loss functions for parametric estimation and the stochastic approximation of the subgradients.

Selecting drift estimation as an example in this section, we set the convex and compact parameter space $\Theta\subset\R^{p}$ and the triple of measurable functions $\left(\bModel,M,J\right)$ such that $\bModel\left(x,\theta\right)$ is the possibly misspecified parametric model, $M\left(x\right)$ is a positive semi-definite matrix-valued function, and $J\left(\theta\right)$ is the regularization term.
We define the following function:
\begin{align*}
    \phi\left(x,y,\theta\right):=\frac{1}{2}M\left(x\right)\left[\left(y-\bModel\left(x,\theta\right)\right)^{\otimes2}\right]+J\left(\theta\right).
\end{align*}
Assume that $\phi\left(x,y,\theta\right)$ is convex in $\theta$ for all $x,y\in\R^{d}$ and has measurable elements in the subdifferential for all $x,y,$ and $\theta$.
$\{\theta_{i};i=1,\ldots,n+1\}$ defined by the following online gradient descent algorithm
\begin{align*}
    \theta_{i+1}:=\Proj_{\Theta}\left(\theta_{i}-\frac{h_{n}}{\sqrt{i}}\partial_{\theta} \phi\left(X_{\left(i-1\right)h_{n}}^{a,b},\frac{1}{h_{n}}\Delta_{i}X^{a,b},\theta_{i}\right)\right),
\end{align*}
with an arbitrary initial value $\theta_{1}\in \Theta$ and a sequence of discrete observations $\{X_{ih_{n}}^{a,b};i=0,\ldots,n\}$, is then well-defined as a sequence of random variables by choosing measurable subgradients, where $\Delta_{i}X^{a,b}=X_{ih_{n}}^{a,b}-X_{\left(i-1\right)h_{n}}^{a,b}$ and $h_{n}>0$ is the discretization step.
Note that the learning rate chosen here does not lead to the best convergence but is simple and approximately the best in our study.
Our contributions (i) and (iii) provide the following risk bound for the estimator $\Bar{\theta}_{n}:=\frac{1}{n}\sum_{i=1}^{n}\theta_{i}$ with a fixed $\left(a,b\right)$: for some $c>0$,
\begin{align*}
    &\sup_{\theta\in\Theta}\left(\E_{x}^{a,b}\left[f^{a,b}\left(\Bar{\theta}_{n}\right)\right]-f^{a,b}\left(\theta\right)\right)\le c\left(\frac{\log nh_{n}^{2}}{\sqrt{nh_{n}^{2}}}+h_{n}^{\beta/2}\right),
\end{align*}
where $\beta\in\left[0,1\right]$ is a parameter controlling the smoothness of $b$, $\E_{x}^{a,b}$ is the expectation over $\{X_{t}^{a,b}\}_{t\ge0}$ with $X_{0}^{a,b}=x$, $f^{a,b}\left(\theta\right)=\int M\left(\xi\right)[(\bModel\left(\xi,\theta\right)-b\left(\xi\right))^{\otimes2}]\Pi^{a,b}\left(\diff \xi\right)+J\left(\theta\right)$ is the loss function, and $\Pi^{a,b}$ is the invariant probability measure of $X_{t}^{a,b}$.
The contribution (ii) yields the existence of $S:=\left\{\left(a,b\right)\right\}$, a class of coefficients of SDEs, and $c$ such that the inequality holds uniformly in $S$; hence, the following uniform risk bound holds:
\begin{align*}
    &\sup_{\left(a,b\right)\in S}\sup_{\theta\in\Theta}\left(\E_{x}^{a,b}\left[f^{a,b}\left(\Bar{\theta}_{n}\right)\right]-f^{a,b}\left(\theta\right)\right)\le c\left(\frac{\log nh_{n}^{2}}{\sqrt{nh_{n}^{2}}}+h_{n}^{\beta/2}\right).
\end{align*}
Note that $\Bar{\theta}_{n}$ estimates the best $\theta\in\Theta$ (or the quasi-optimal parameter; \citealp[see][]{UY11}) with $\bModel\left(\cdot,\theta\right)$ closest to the true $b\left(\cdot\right)$ in the $L^{2}\left(\Pi^{a,b}\right)$-distance.
Moreover, if the model $\bModel$ correctly specifies $b$, that is, for all $\left(a,b\right)\in S$ there exists $\theta$ such that $b\left(\cdot\right)=\bModel\left(\cdot,\theta\right)$, then $f^{a,b}\left(\theta\right)=0$ and we obtain the uniform upper bound of $\E_{x}^{a,b}\left[f^{a,b}\left(\Bar{\theta}_{n}\right)\right]$.
% \begin{align*}
%     \sup_{\left(a,b\right)\in S}\sup_{\theta\in\Theta}\E_{x}^{a,b}\left[f^{a,b}\left(\Bar{\theta}_{n}\right)\right]\le c\left(\frac{\log nh_{n}^{2}}{\sqrt{nh_{n}^{2}}}+h_{n}^{\beta/2}\right).
% \end{align*}

One simple but significant outcome of the above discussion is a non-asymptotic risk guarantee of the following online gradient descent for linear models such that an arbitrary initial value $\theta_{1}\in \Theta$,
\begin{align*}
    \theta_{i+1}:=\Proj_{\Theta}\left(\theta_{i}+\frac{1}{\sqrt{i}}\left(\partial_{\theta}\bModel\left(X_{\left(i-1\right)h_{n}}^{a,b},\theta_{i}\right)\right)\left(\Delta_{i}X^{a,b}-h_{n}\bModel\left(X_{\left(i-1\right)h_{n}}^{a,b},\theta_{i}\right)\right)\right),
\end{align*}
where $\bModel\left(x,\theta\right)$ is the possibly misspecified parametric model whose components are linear in $\theta\in\Theta$.
Note that it corresponds to the case $M\left(x\right)=I_{d}$, $J\left(\theta\right)=0$.
As evident, the uniform risk bound for the estimator $\Bar{\theta}_{n}:=\frac{1}{n}\sum_{i=1}^{n}\theta_{i}$ over a certain family $S$ of the coefficients $a,b$ holds: for some $c>0$,
\begin{align*}
    &\sup_{\left(a,b\right)\in S}\sup_{\theta\in\Theta}\left(\E_{x}^{a,b}\left[\int \left\|\bModel\left(\xi,\Bar{\theta}_{n}\right)-b\left(\xi\right)\right\|_{2}^{2}\Pi^{a,b}\left(\diff \xi\right)\right]-\int \left\|\bModel\left(\xi,\theta\right)-b\left(\xi\right)\right\|_{2}^{2}\Pi^{a,b}\left(\diff \xi\right)\right)\\
    &\qquad\le c\left(\frac{\log nh_{n}^{2}}{\sqrt{nh_{n}^{2}}}+h_{n}^{\beta/2}\right).
\end{align*}
If $\bModel$ correctly specifies $b$, then
\begin{align*}
    \sup_{\left(a,b\right)\in S}\E_{x}^{a,b}\left[\int \left\|\bModel\left(\xi,\Bar{\theta}_{n}\right)-b\left(\xi\right)\right\|_{2}^{2}\Pi^{a,b}\left(\diff \xi\right)\right]\le c\left(\frac{\log nh_{n}^{2}}{\sqrt{nh_{n}^{2}}}+h_{n}^{\beta/2}\right).
\end{align*}

Table \ref{tab:literature} presents a comparison of the estimator $\Bar{\theta}$ with those derived via other methods (we do not compare the diffusion estimation methods because the rates of convergence depend on model specification, and thus, the comparison is not straightforward; see \citealp{UY11}).
Our online estimation method has several advantages: the convergence under misspecification is guaranteed, uniform risk bounds are derived, computational complexities in online estimation are lower, and non-differentiable $a$ and $b$ with respect to $x$ are allowed.
The convergence rate of $\Bar{\theta}$ is not optimal \citep[see][]{Gob02}; however, this is not peculiar because the optimal rate by \citet{Gob02} is obtained for $a$ and $b$ with some degrees of smoothness, which we do not assume (a small improvement by assuming twice-differentiability is immediate by It\^{o}'s formula; see Section 4).
Rather, note that we derive the consistency and rate of convergence without using explicit unbiased predictions for increments \citep{BS95,KS99} or second-order differentiability for the It\^{o}--Taylor expansion to obtain the uniform law of large numbers \citep[see][]{Flo89,Kes97} via the argument by \citet{IH81}.
%It is typical to allow at least twice continuous differentiability for $b$ because we can use It\^{o}--Taylor expansions for pointwise law of large numbers of contrast fucntions \citep[e.g., see][]{Flo89,Kes97} which is necessary to show uniform law of large numbers via the argument by \citet{IH81}.

%The estimator proposed in our study does not have the optimal rate of convergence $1/\sqrt{nh_{n}}$ \citep{Gob02}, it does allow misspecified modelling as well as nonlinearity and is computationally efficient in the construction of the estimator.
%Note that estimation of SDEs with misspecified models itself can be a topic of interest \citep[for details, see][]{UY11,Ogi21}, and deterioration of convergence rates under model misspecification is also discussed in the previous studies.
% Note that the estimators which can achieve optimal rates of convergence under correct model specification have worse rate of convergence under model misspecification; see \citet{UY11,Ogi21}.
%In addition, one can see motivations to adopt our estimation methods to guarantee uniform risk bounds and compute with low complexity even if we assume linearity of models.
\begin{table}[ht]
    %\captionsetup{singlelinecheck = false, format= hang, justification=raggedright, font=footnotesize, labelsep=space}
    \caption{\small Comparison of our result on $p$-dimensional parameter estimation of the drift coefficient $b$ with discrete observations and fixed sample size $n$ with the results of other studies.
    We term computational complexity to estimate $m$ times based on $\left\{X_{ih_{n}};i=0,\ldots,k\left[n/m\right]\right\}$ for all $k=1,\ldots,m$ with $m=1,\ldots,n$ as the ``complexity in online estimation''.
    We set $\mathcal{O}\left(p^{3}\right)$ as the complexity of matrix inversion.
    The evaluation of the least squares estimation (LSE) for linear SDE models is based on the study by \citet{UY11} and an additional assumption that the diffusion coefficient is the identity matrix.
    The complexity for the approach of \citet{KUY17} is given if we ignore the computational complexity of the initial Bayes estimator.
    $\mathcal{O}\left(np\right)$ in line with the approach of \citet{Vas14} is obtained by assuming the independence of each component of the SDEs because the result is applicable for one-dimensional Ornstein--Uhlenbeck processes.
    The complexity of our estimation is achieved by assuming simple $\Theta$, letting the projection on $\Theta$ be $\mathcal{O}\left(p\right)$.
    ``Non-smooth coefficients are allowed'' implies that $a$ and $b$ can be non-differentiable with respect to $x$.
    Note that our method's convergence rate is obtained by choosing $\eta_{i}:=h_{n}/\sqrt{i\log nh_{n}^{2}}$ and assuming a usual identifiability condition: see Section 4.
    }
    \tiny
    \begin{tabular}{l|l|l|l|l|l}
        \multirow{3}{*}{Method} & misspecified & uniform & complexity  & non-smooth & convergence rate  \\
        & modelling & risk bounds & in $m$-times & coefficients & of the estimator \\
        & is allowed & are shown & online estimation & are allowed & ($1/\sqrt{nh_{n}}$ is optimal) \\\hline\hline
        LSE for linear SDE models & \multirow{2}{*}{\cmark} & \multirow{2}{*}{\xmark} & \multirow{2}{*}{$\mathcal{O}\left(np^{2}+mp^{3}\right)$} & \multirow{2}{*}{\xmark} & \multirow{2}{*}{$1/\sqrt{nh_{n}}$}\\
        \citep[using][]{UY11} & & & & & \\\hline
        Hybrid estimation  & \multirow{2}{*}{\xmark} & \multirow{2}{*}{\xmark} &  \multirow{2}{*}{$\mathcal{O}\left(m\left(n+p\right)p^{2}\right)$} & \multirow{2}{*}{\xmark} & \multirow{2}{*}{$1/\sqrt{nh_{n}}$}\\
        \citep{KUY17} & & & & & \\\hline
        Truncated estimation & \multirow{2}{*}{\xmark} & \multirow{2}{*}{\cmark} & \multirow{2}{*}{$\mathcal{O}\left(np\right)$ (if diagonal)} & \multirow{2}{*}{\xmark} & \multirow{2}{*}{$1/\sqrt{nh_{n}^{2}}$} \\
        \citep{Vas14} & & & &  & \\\hline
        %LSE  & linear & not allowed & $\mathcal{O}\left(np^{2}+mp^{3}\right)$ & $1/\sqrt{nh_{n}}$ \\
        Online gradient descent & \multirow{2}{*}{\cmark} & \multirow{2}{*}{\cmark} & \multirow{2}{*}{$\mathcal{O}\left(np\right)$} &\multirow{2}{*}{\cmark} &  \multirow{2}{*}{$\sqrt[4]{\frac{\log nh_{n}^{2}}{nh_{n}^{2}}+h_{n}^{\beta}}$}\\
        (this study) & & & & & \vspace{5mm}
        % & & & & & 
    \end{tabular}
    \label{tab:literature}
\end{table}

To allow the risk bound and estimator to be convergent, it is necessary to assume $nh_{n}^{2}\to\infty$ and $h_{n}\to0$ as $n\to\infty$.
This seems to be a peculiar assumption compared to $nh_{n}\to\infty$ and $nh_{n}^{p}\to0$ for some $p\ge2$, which is often assumed in statistics of SDEs.
This is because of the properties of approximations for subgradients of loss functions; if $h_{n}$ is smaller, the bias of the approximation is smaller, but the variance is greater, and vice versa.
Therefore, the frequency of observations should be moderate---neither too low nor too high---to balance the two risk bound terms.
Notably, other estimation methods may not satisfy their regularity conditions under moderate frequency observations, and the comparison may not be straightforward.

In addition to the uniform risk guarantees of the proposed online estimation, this study also contributes to the two technical topics (1) gradient-based estimation of dependent data and (2) the simultaneous ergodicity of classes of SDEs, which are used for the derivation of the guarantees.
Our result generalizes an existing result on SMD with dependent subgradients by considering the approximations of the subgradients, which are essential in the estimation of SDEs with discrete observations.
We relax the conditions for the simultaneous ergodicity of a class of SDEs, whereas the previous studies which obtain the concise representation of constants assume strong conditions such as $d=1$ and smoothness of $a$ and $b$.
Our results contribute to three statistical topics (3) estimation with non-asymptotic uniform risk guarantees, (4) computationally efficient one, and (5) misspecified modelling of SDEs as well.
This study adds new results to the derivation of non-asymptotic uniform risk bounds, which has also been studied in the sequential estimation for discretely observed diffusion processes.
The computational complexity of our estimation method is $\mathcal{O}\left(np\right)$ if the complexity of $\Proj_{\Theta}$ is $\mathcal{O}\left(p\right)$, which is minimal in online estimation and lower than other methods achieving consistency.
Our study is also a novel approach to analysing misspecified models using convex optimization, whereas the existing studies are based on analyses of quasi-log-likelihood functions.

\subsection{Outline}
Sections 2--4 demonstrate our contributions with respect to optimization, probability, and statistics separately.
Section 2 presents the convergence guarantees of SMD with biased and dependent subgradients, which extends the discussion of \citet{DAJ+12}.
In Section 3, we consider simultaneous ergodicity and moment bounds based on the recent result on Aronson-type estimates for the transition density functions of SDEs \citep{MPZ21}.
Section 4 includes the main result on the online parametric estimation of SDEs, applying a classical discussion in the estimation of diffusion processes \citep[e.g., see][]{Kes97}.

\subsection{Notation}
$\left\|\cdot\right\|_{\ast}$ is the dual norm of a norm $\left\|\cdot\right\|$.
For any convex function $f:\R^{p}\to\R$, $\partial f\left(a\right):=\left\{b\in\R^{p};f\left(x\right)\ge f\left(a\right)+\left\langle b,x-a\right\rangle,\ \forall x\in\R^{p}\right\}$ is the subdifferential of $f$ at $a\in\R^{p}$.
For any vector $x\in\R^{\ell}$ with $\ell\in\N$ and $p\in\left[1,\infty\right]$, $\left\|x\right\|_{p}$ denotes $\ell^{p}$-norm for vectors.
For arbitrary matrix $A$, $\left\|A\right\|_{2}$ and $\left\|A\right\|_{F}$ denote the spectral and Frobenius norms.
%$\left\llbracket a,b\right\rrbracket:=\left\{k\in\Z;a\le k\le b\right\}$ with $a,b\in \Z$.
$\left\|\mu\right\|_{\TV}:=\mu_{+}\left(\mathbf{X}\right)+\mu_{-}\left(\mathbf{X}\right)$ is the total variation norm of any finite signed measure $\mu$ on a measurable space $\left(\mathbf{X},\mathcal{X}\right)$, where $\mu=\mu_{+}-\mu_{-}$ is the Hahn decomposition for $\mu$.
For closed $C\subset\R^{p}$, $\Proj_{C}\left(x\right)$ denotes a projection of $x\in\R^{p}$ onto $C$.
For any matrix $A$, $A^{\top}$ is the transpose of $A$, and $A^{\otimes2}:=AA^{\top}$.
For any two matrices $A,B$ with the same size, $A\left[B\right]=\tr\left(A^{\top}B\right)$.

\section{Stochastic mirror descent with dependence and bias}
Our first result is an extension of that by \citet{DAJ+12}, which discusses the SMD algorithm with dependent noises.
Specifically, we provide convergence guarantees for the SMD algorithm based on the approximated subgradients of latent loss functions dependent on ergodic noises, which is necessary to view the convergence rate of our estimators discussed in Section 4.
First, we review the problem setting of \citet{DAJ+12} to understand the reason that an extension is needed.
Subsequently, we provide the theoretical convergence guarantees for the SMD algorithm based on approximated subgradients.

\subsection{Motivation}
\citet{DAJ+12} consider the minimization problem for the convex loss function $f\left(\theta\right)$ defined as 
\begin{align*}
    f\left(\theta\right):=\E\left[F\left(\theta;\xi\right)\right]=\int_{\Xi}F\left(\theta;\xi\right)\Pi\left(\diff \xi\right),
\end{align*}
where $\theta\in\Theta$, $\Theta\subset\R^{p}$ is a compact convex set, $\xi$ is a random variable whose distribution is given as $\Pi$, $\Xi$ is the state space of $\xi$, and $\left\{F\left(\cdot;\xi\right);\xi\in\Xi\right\}$ is a family of convex functions.
They show convergence in expectation and with high probability of the SMD algorithm using the subgradient of the sampled loss functions $\left\{F\left(\cdot;\xi_{i}\right);i=1,\ldots,n\right\}$, where $n$ is the sample size and $\left\{\xi_{i}\right\}$ is an ergodic process whose invariant probability measure is $\Pi$.

In statistical estimation, a loss function of interest sometimes depends on the true values of the unknown parameters.
Considering the estimation of i.i.d.\ random variables or discrete-time stochastic processes, we often obtain an equivalent optimization problem with another loss function and sampled version, whose gradients depend only on observations and are independent of true values.
For instance, let us consider a stochastic regression $y_{i}=x_{i}^{\top}\theta+\varepsilon_{i}$, where $\{x_{i}\}$ and $\{\epsilon_{i}\}$ are square-integrable centred stationary mixing processes independent of each other.
If the loss function $f(\theta)$ is defined as $f(\theta):=\E[(x_{i}^{\top}\theta-x_{i}^{\top}\theta^{\star})^{2}]$ with $\theta^{\star}$ denoting the true value of $\theta$, then $\E[(x_{i}^{\top}\theta-x_{i}^{\top}\theta^{\star})^{2}]=\E[(x_{i}^{\top}\theta-y_{i})^{2}]-\E[\varepsilon_{i}^{2}]=:f'(\theta)$.
If we let $F'(\theta;\xi_{i})=(x_{i}^{\top}\theta-y_{i})^{2}-\varepsilon_{i}^{2}$, the gradient $\partial F'=2\theta(x_{i}^{\top}\theta-y_{i})$ is dependent on observations and independent of $\theta^{\star}$.
Hence, we can guarantee the convergence of the SMD algorithm based on $\partial F'$ for both $f(\theta)$ and $f'(\theta)$ in expectation.% {\color{red} and with high probability}.

However, in the estimation of diffusion processes based on discrete observations, considering an equivalent optimization problem independent of the true values of the parameters is difficult.
Alternatively, we can observe approximated subgradients of sampled loss functions based on discrete observations and independent of the true values.
Hence, we examine the convergence of the SMD algorithm based on approximate subgradients, which leads to the convergence of the estimation of diffusion processes via online gradient descent.

As an intuitive example of approximation, let us consider the estimation of an ergodic diffusion process $\left\{X_{t};t\ge0\right\}$ defined by a parametric SDE $\diff X_{t}=b\left(X_{t},\theta\right)\diff t+a\left(X_{t}\right)\diff w_{t},\ X_{0}=x_{0}$ with the invariant probability measure $\Pi$ and $b\left(x,\theta\right)$, whose elements are linear in $\theta$ for all $x$.
We estimate $\theta$ based on discrete observations $\left\{X_{ih_{n}};i=0,\ldots,n\right\}$ with the discretization step $h_{n}>0$.
A typical loss function for estimation of $\theta$ is the $L^{2}$-loss function $f\left(\theta\right)=\int \left\|b\left(x,\theta\right)-b\left(x,\theta^{\star}\right)\right\|_{2}^{2}\Pi\left(\diff x\right)$, with the true value $\theta^{\star}$, and its sampled version is $F\left(\theta;\xi_{i}\right)=\left\|b\left(X_{\left(i-1\right)h_{n}},\theta\right)-b\left(X_{\left(i-1\right)h_{n}},\theta^{\star}\right)\right\|_{2}^{2}$, with $\xi_{i}=X_{\left(i-1\right)h_{n}}$.
As it is difficult to provide functions based on the observations and independent of $\theta^{\star}$, whose optimization is equivalent to $f\left(\theta\right)$, we consider an approximately equivalent problem independent of $\theta^{\star}$.
The following holds:
\begin{align*}
    &\left\|b\left(X_{\left(i-1\right)h_{n}},\theta\right)-b\left(X_{\left(i-1\right)h_{n}},\theta^{\star}\right)\right\|_{2}^{2}\\
    &=\frac{1}{h_{n}^{2}}\left\|\Delta_{i}X-h_{n}b\left(X_{\left(i-1\right)h_{n}},\theta\right)\right\|_{2}^{2}-\frac{1}{h_{n}^{2}}\left\|\Delta_{i}X-h_{n}b\left(X_{\left(i-1\right)h_{n}},\theta^{\star}\right)\right\|_{2}^{2}\\
    &\quad+\frac{2}{h_{n}}\left\langle b\left(X_{\left(i-1\right)h_{n}},\theta\right)-b\left(X_{\left(i-1\right)h_{n}},\theta^{\star}\right),\Delta_{i}X-h_{n}b\left(X_{\left(i-1\right)h_{n}},\theta^{\star}\right)\right\rangle,
\end{align*}
where $\Delta_{i}X:=X_{ih_{n}}-X_{\left(i-1\right)h_{n}}$.
Note that the second term is independent of $\theta$, and the third term on the right-hand side is negligible under mild conditions when $h_{n}$ is small; therefore, we expect that the SMD algorithm based on an observable sequence of random functions $\left\|\Delta_{i}X-h_{n}b\left(X_{\left(i-1\right)h_{n}},\theta\right)\right\|_{2}^{2}/h_{n}^{2}$ should approximately optimize $f\left(\theta\right)$.
% We may think of another approach to applying the result of \citet{DAJ+12} to $\left\|\Delta_{i}X-h_{n}b\left(X_{\left(i-1\right)h_{n}},\theta\right)\right\|_{2}^{2}$ directly; however, the approximation approach has an advantage that we only need to evaluate mixing properties of $X_{t}$ rather than vectorized $\left(X_{ih_{n}},X_{\left(i-1\right)h_{n}}\right)$ with $h_{n}\downarrow 0$.

Notably, we need not examine the mixing properties of the approximate subgradients, which are sometimes not obvious in the estimation of diffusion processes based on discrete and partial observations.
Approximating contrast functions using latent diffusion processes is quite common in studies on the estimation of diffusion processes based on partial observations such as integrated \citep{Glo00,Glo06} and noisy observations \citep{Fav14,Fav16}; therefore, we can expect the convergence guarantees for SMD with approximated subgradients to provide a simple but useful tool to analyse online estimation methods based on various observation schemes.

\subsection{Stochastic mirror descent algorithm and the key decomposition}
We state the problem and propose the SMD algorithm with approximate subgradients.
$\left(\Omega,\mathcal{A},P\right)$ denotes the probability space.
$\left(\Xi,\mathcal{B}\left(\R^{d}\right)|_{\Xi}\right)$ with $\Xi\in\mathcal{B}\left(\R^{d}\right)$ is the state space of a latent ergodic process $\left\{\xi_{i};i\in\N\right\}$ with the invariant probability measure $\Pi$ on $\left(\Xi,\mathcal{B}\left(\R^{d}\right)|_{\Xi}\right)$.
% $\left(\Xi,\mathcal{B}\left(\R^{d}\right)|_{\Xi}\right)$ with $\Xi\in\mathcal{B}\left(\R^{d}\right)$ is the state space of a latent process $\left\{\xi_{i}\right\}$, and $\Pi$ is the probability measure on $\left(\Xi,\mathcal{B}\left(\R^{d}\right)|_{\Xi}\right)$.
We set a compact and convex set $\Theta\in \mathcal{B}\left(\R^{p}\right)$ as the parameter space with the norm $\|\cdot\|$.

% original problem
Let $\left\{F\left(\cdot;\xi\right);\xi\in\Xi\right\}$ be a family of real-valued convex functions defined on $N_{\Theta}$, where $N_{\Theta}$ is an open neighbourhood of $\Theta$.
%a.s.\ is removed because there are countably many measures for $\xi_{i}$
%Let $\left\{F\left(\cdot;\xi\right);\xi\in\Xi\right\}$ be a family of $P$-a.s.\ real-valued closed convex function defined on $N_{\cTheta}$, where $N_{\cTheta}$ is an open neighbourhood of $\cTheta$.
We assume a convex function $f$ such that
\begin{align*}
    f\left(\theta\right):=\int_{\Xi}F\left(\theta;\xi\right)\Pi\left(\diff \xi\right)
\end{align*}
is finite-valued for all $\theta\in N_{\Theta}$.
We consider the following minimization problem:
\begin{align*}
    \min_{\theta\in\Theta}f\left(\theta\right).
\end{align*}

% subgradients
We let $\partial F\left(\theta;\xi\right)$ denote the subdifferential of $F$ with respect to $\theta$ and assume that there exists a $\left(\mathcal{B}\left(\R^{d}\right)|_{\Xi}\right)\otimes\left(\mathcal{B}\left(\R^{p}\right)|_{N_{\Theta}}\right)$-measurable function $\sfG\left(\theta;\xi\right)$ such that $\sfG\left(\theta;\xi\right)\in\partial F\left(\theta;\xi\right)$ for all $\theta\in\Theta$ and $\xi\in\Xi$.

% Bregman divergence
We set a prox-function $\psi$, a differentiable $1$-strongly convex function on $N_{\Theta}$ with respect to the norm $\left\|\cdot\right\|$.
$D_{\psi}$ is the Bregman divergence generated by $\psi$ such that for all $\theta,\theta'\in\Theta$,
\begin{align*}
    D_{\psi}\left(\theta,\theta'\right):=\psi\left(\theta\right)-\psi\left(\theta'\right)-\left\langle \nabla \psi\left(\theta'\right),\theta-\theta'\right\rangle \ge \frac{1}{2}\left\|\theta-\theta'\right\|^{2}.
\end{align*}
We assume $\sup_{\theta_{1},\theta_{2}\in\Theta}D_{\psi}\left(\theta_{1},\theta_{2}\right)\le R^{2}/2$ for some $R>0$.
Note that SMD is equivalent to stochastic gradient descent if $\psi=\|\cdot\|_{2}^{2}/2$ and $\|\cdot\|=\|\cdot\|_{\ast}=\|\cdot\|_{2}$.

% Let $\left\{\xi_{i};i\in\N\right\}$ be a $\Xi$-valued mixing process with the invariant probability measure $\Pi$.
% Let $\left\{\xi_{i};i\in\N\right\}$ is a stochastic process on $\Xi$, and we call it as the reference process and expect that it is mixing.
We consider the SMD algorithm based on the gradients of the approximating functions $H_{i,n}\left(\cdot\right)$ for $F\left(\cdot;\xi_{i}\right)$.
% SMD with approximation
Let $\left\{H_{i,n}\left(\cdot\right);i=1,\ldots,n\right\}$ be a sequence of real-valued random convex functions on $N_{\Theta}$. 
Assume that there exists an $\mathcal{A}\otimes\left(\mathcal{B}\left(\R^{p}\right)|_{N_{\Theta}}\right)$-measurable random function $\sfK_{i,n}\left(\theta\right)$ such that $\sfK_{i,n}\left(\theta\right)\in\partial H_{i,n}\left(\theta\right)$ almost surely (a.s.) for all $\theta\in\Theta$.
We define the SMD update: for all $i=1,\ldots,n$ and arbitrary chosen $\theta_{1}\in\Theta$,
\begin{align}
    \theta_{i+1}=\argmin_{\theta\in\Theta}\left\{\left\langle \sfK_{i,n}\left(\theta\right),\theta\right\rangle+\frac{1}{\eta_{i}}D_{\psi}\left(\theta,\theta_{i}\right)\right\},\label{SMDUpdate}
\end{align}
where $\left\{\eta_{i}\right\}$ is a sequence of non-increasing positive numbers denoting learning rates.

% residual
Let $\resid_{\tau,n}$ with $\tau\in\N_{0}\left(:=\N\cup\left\{0\right\}\right)$ be a random function of a sequence of $\Theta$-valued random variables $\left\{\vartheta_{i}\right\}$ such that
\begin{align}
    \resid_{\tau,n}\left(\left\{\vartheta_{i}\right\}\right)&:=\sum_{i=1}^{n-\tau}\left(F\left(\vartheta_{i};\xi_{i+\tau}\right)-H_{i+\tau,n}\left(\vartheta_{i}\right)\right);
\end{align}
we use the abbreviation $\resid_{\tau,n}\left(\theta^{\prime}\right):=\resid_{\tau,n}\left(\left\{\theta^{\prime}\right\}\right)$ for non-random $\theta'\in\Theta$.
$\resid_{\tau,n}$ measures the degrees of discrepancy between $F\left(\cdot;\xi_{i}\right)$ and $H_{i,n}\left(\cdot\right)$.

The following decomposition for $\tau\in\N_{0}$ is useful:
\begin{align}
    \sum_{i=1}^{n}\left(f\left(\theta_{i}\right)-f\left(\theta^{\prime}\right)\right)
    &=\sum_{i=1}^{n-\tau}\left(f\left(\theta_{i}\right)-f\left(\theta^{\prime}\right)-F\left(\theta_{i};\xi_{i+\tau}\right)+F\left(\theta^{\prime};\xi_{i+\tau}\right)\right)\notag\\
    &\quad+\sum_{i=1}^{n-\tau}\left(H_{i+\tau,n}\left(\theta_{i}\right)-H_{i+\tau,n}\left(\theta_{i+\tau}\right)\right)+\sum_{i=\tau+1}^{n}\left(H_{i,n}\left(\theta_{i}\right)-H_{i,n}\left(\theta^{\prime}\right)\right)\notag\\
    &\quad+\sum_{i=n-\tau+1}^{n}\left(f\left(\theta_{i}\right)-f\left(\theta^{\prime}\right)\right)+\resid_{\tau,n}\left(\left\{\theta_{i}\right\}\right)-\resid_{\tau,n}\left(\theta^{\prime}\right),\label{eq:SMDDec}
\end{align}
which is a trivial extension to (6.2) by \citet{DAJ+12}.

\subsection{Convergence in expectation}
% Hellinger distance
We define the Hellinger distance between two probability measures $P$ and $Q$ defined on the common measurable space such that
\begin{align}
    d_{\mathrm{Hel}}\left(P,Q\right):=\sqrt{\int \left(\sqrt{\frac{\diff P}{\diff \mu}}-\sqrt{\frac{\diff Q}{\diff \mu}}\right)^{2}\diff \mu},
\end{align}
where $\mu$ is a measure such that $P$ and $Q$ are absolutely continuous with respect to $\mu$.
Such a $\mu$ exists; for example, $P$ and $Q$ are absolutely continuous with respect to $\frac{1}{2}\left(P+Q\right)$.

% Filtration
Let us consider that $\F:=\left\{\calF_{i};i\in\N_{0}\right\}$ is a filtration such that $\sigma\left(\xi_{j};j\le i\right)\subset \calF_{i}$ for all $i\in\N_{0}$ and $\sigma\left(\theta_{j};j\le i+1\right)\subset \calF_{i}$ for all $i=0,\ldots,n$.
Note that $\calF_{i}$-measurability of $\theta_{i+1}$ is natural because $\theta_{i+1}$ depends on $\xi_{1},\ldots,\xi_{i}$ if we do not consider the approximation of $F\left(\cdot;\xi_{i}\right)$ with $H_{i,n}\left(\cdot\right)$.
We do not determine a concrete $\F$ because appropriate selection depends on applications.

% Mixing times
We define the mixing time for $\xi_{i}$ with respect to the Hellinger distance based on the filtration $\F$: $P_{\left[i\right]|\F}:=\left\{P_{\left[i\right]|\F}^{j};j> i\right\},\ i\in\N_{0}$ which denotes a family of $P_{\left[i\right]|\F}^{j}$, the conditional distribution of $\xi_{j}$ given $\calF_{i}$ with $j>i$, and 
\begin{align}
    \tau\left(P_{\left[i\right]|\F},\epsilon\right)&:=\inf\left\{\tau\in\N;d_{\mathrm{Hel}}^{2}\left(P_{\left[i\right]|\F}^{i+\tau},\Pi\right)\le \epsilon^{2}\right\}.
\end{align}
\begin{remark}
Note that we only consider the Hellinger distance for the analysis of the SMD algorithm, whereas \citet{DAJ+12} also present the result for the total variation distance under the assumption that the dual norms of the subgradients are a.s.\ bounded by a positive constant.
It is because our analysis is based on the uniform boundedness of the expectation of the squared dual norms of subgradients.
This assumption is even weaker than that of \citet{DAJ+12} for the result using the Hellinger distance, which assumes the uniform a.s.\ boundedness of the conditional expectation of the squared dual norms with respect to the filtration.
\end{remark}
% It is natural: we imagine $\theta_{i-1}$, which is determined by X_{i-2}, and $\xi_{i}$ is $X_{\left(i-1\right)h_{n}}$ and it indeed advancely measure: in diffusion processes, it not so significant because we bound expectation of squared norm of $\sfG$, not squared norm of $\sfG$.
Let us present some assumptions.
\begin{enumerate}
    \item[(A1)] There exists a constant $G>0$ such that for all $i\in\N$, $\calF_{i\wedge n-1}$-measurable $\Theta$-valued random variable $\vartheta_{i}$,
    \begin{align*}
        \E\left[\left\|\sfG\left(\vartheta_{i};\xi_{i}\right)\right\|_{\ast}^{2}\right]\le G^{2}.
    \end{align*}
    \item[(A2)] There exists a constant $K_{n}>0$ such that for all $i=1,\ldots,n$, $\calF_{i-1}$-measurable $\Theta$-valued random variables $\vartheta_{i}$, 
    \begin{align*}
        \E\left[\left\|\sfK_{i,n}\left(\vartheta_{i}\right)\right\|_{\ast}^{2}\right]\le K_{n}^{2}.
    \end{align*}
    \item[(A3)] The mixing times of $\left\{\xi_{i}\right\}$ are uniform in the sense that there exists a uniform mixing time in expectation $\tau_{\E}\left(P_{|\F},\epsilon\right)<\infty$ such that for all $\epsilon>0$,
    \begin{align*}
        \tau_{\E}\left(P_{|\F},\epsilon\right):=\inf\left\{\tau\in\N;\sup_{i\in\N_{0}}\E\left[d_{\mathrm{Hel}}^{2}\left(P_{\left[i\right]|\F}^{i+\tau},\Pi\right)\right]\le \epsilon^{2}\right\}.
    \end{align*}
    For simplicity, we ignore the dependence of $\tau_{\E}$ on $P_{|\F}$ and use the notation $\tau_{\E}\left(\epsilon\right)$.
\end{enumerate}
Note that the different ranges of $i$ in (A1) and (A2) are not essential; we need (A1) with $i\in\N$ to evaluate the expectation with respect to $\Pi$ by Fatou's lemma (see Lemma \ref{lem:op:DAJ+12L63}).
% We also assume $F$ is continuously differentiable and at most polynomial growth.

% \begin{remark}
% When choosing $\calF_{i}=\sigma\left(\sfK_{j,n};j\le i-1\right)$, $\theta_{i}$ is a $\calF_{i}$-adapted process.
% All the results depend on the choice of $\left\{\calF_{i}\right\}$.
% The choice of $\left\{\calF_{i}\right\}$ is significant when we consider the parametric estimation of partially observed stochastic processes.
% \end{remark}

We obtain a version of Theorem 3.1 by \citet{DAJ+12}.

\begin{theorem}\label{thm:op:SMDEx}% [a version of Theorem 3.1 by \citealp{DAJ+12}]
Under (A1)--(A3), for any $\epsilon>0$ and $\theta^{\prime}\in\Theta$,
\begin{align*}
    \E\left[\sum_{i=1}^{n}\left(f\left(\theta_{i}\right)-f\left(\theta^{\prime}\right)\right)\right]&\le 2\sqrt{2}GRn\epsilon+\sqrt{2}\left(\tau_{\E}\left(\epsilon\right)-1\right)K_{n}^{2}\sum_{i=1}^{n}\eta_{i}+\frac{R^{2}}{2\eta_{n}}+\frac{K_{n}^{2}}{2}\sum_{i=1}^{n}\eta_{i}\\
    &\quad+\left(\tau_{\E}\left(\epsilon\right)-1\right)GR+\E\left[\resid_{\tau_{\E}\left(\epsilon\right)-1,n}\left(\left\{\theta_{i}\right\}\right)-\resid_{\tau_{\E}\left(\epsilon\right)-1,n}\left(\theta^{\prime}\right)\right].
\end{align*}
\end{theorem}
The proof is presented in the Appendix.
This upper bound is the same as that in Theorem 3.1 by \citet{DAJ+12} except for the residuals, which immediately disappear if $F\left(\cdot;\xi_{i}\right)=H_{i,n}\left(\cdot\right)$, and the constant factor $\sqrt{2}$ of the second term on the right hand side.
Assumptions (A1) and (A2) on the subgradients are weaker than Assumption A in their study; therefore, this result includes a generalization of that by \citet{DAJ+12} in the sense of achieving the same bound except for the constant factor with a weaker condition.

\section{Simultaneous ergodicity of classes of diffusion processes}
We discuss the simultaneous ergodicity of a family of diffusion processes $X_{t}^{a,b}\left(x\right)$, defined by the following SDE:% defined as the unique weak solution of the following SDE:
\begin{align}
    \diff X_{t}^{a,b}\left(x\right)=b\left(X_{t}^{a,b}\left(x\right)\right)\diff t+a\left(X_{t}^{a,b}\left(x\right)\right)\diff w_{t},\ X_{0}^{a,b}\left(x\right)=x,\ t\ge0,
\end{align}
where $b:\R^{d}\to\R^{d}$ and $a:\R^{d}\to\R^{d}\otimes \R^{d}$ are non-random functions, $x\in\R^{d}$ is a non-random vector, and $w_{t}$ is a $d$-dimensional Wiener process.
The transition kernel is denoted as $P_{t}^{a,b}:\R^{d}\times\mathcal{B}\left(\R^{d}\right)\to\left[0,1\right]$ for all $t>0$. %, and $\left(a,b\right)\in S$ for a family of coefficients $S$.
For simplicity, we occasionally use the notation $X_{t}^{a,b}=X_{t}^{a,b}\left(x\right)$ when no confusion can arise.

In this section, we illustrate the simultaneous ergodicity and uniform moment bounds of a family of diffusion processes.
% The simultaneous ergodicity of a family of diffusion processes refers to the ergodicity such that the rate of convergence is uniform in the family.
% In this section, we illustrate the simultaneous ergodicity of a family of diffusion processes, that is, the rate of convergence becomes uniform in a family of diffusion processes, in addition to uniform moment bounds.
They enable us to validate that the risk bounds by Theorem \ref{thm:op:SMDEx} hold uniformly in families with such properties.

Ergodicity is one of the classical topics in the study of diffusion processes \citep[see][]{Ver88,Ver97,PV01,Kul17}.
\citet{GP14} study the simultaneous exponential ergodicity for a class of Markov chains with uniform constants having concise representations and provide a sufficient condition for the class of one-dimensional diffusion processes with such constants.
Notably, \citet{Kul09} also considers the simultaneous ergodicity of a class of SDEs with jump noises.
We discuss such uniform constants with the recent Aronson-type estimates for the transition density functions of multidimensional SDEs \citep{MPZ21} and classical sufficient conditions for ergodicity \citep{Kul17}.
\citet{MPZ21} illustrate the Aronson-type estimates, whose constants are determined by the parameters in the assumptions regarding the drift and diffusion coefficients, the terminal of the estimates, and the dimension of the process.
Demonstrating the exponential ergodicity of a diffusion process based on Aronson-type estimates is not a novel idea \citep[e.g., see][]{Ver21}; however, the estimates by \citet{MPZ21} enable us to show that the convergence of total variation distances is uniform for a class of SDEs satisfying the same assumptions with the uniform constants.

\subsection{Local Dobrushin condition}
For the local Dobrushin condition, we set the following time-homogeneous versions of the conditions in \citet{MPZ21}.
\begin{itemize}
    \item[($H_{\alpha}^{a}$)] There exist constants $\kappa_{0}\ge1$ and $\alpha\in\left(0,1\right]$ such that for all $x,y,\xi\in\R^{d}$
    \begin{align*}
        \kappa_{0}^{-1}\left\|\xi\right\|_{2}^{2}\le \langle a^{\otimes2}\left(x\right)\xi,\xi\rangle \le \kappa_{0}\left\|\xi\right\|_{2}^{2},
    \end{align*}
    and
    \begin{align*}
        \left\|a\left(x\right)-a\left(y\right)\right\|_{F}\le \kappa_{0}\left\|x-y\right\|_{2}^{\alpha}.
    \end{align*}
    \item[($H_{\beta}^{b}$)] $b$ is measurable, and there exist constants $\kappa_{1}>0$ and $\beta\in\left[0,1\right]$ such that for all $x,y\in\R^{d}$,
    \begin{align*}
        \left\|b\left(0\right)\right\|_{2}\le \kappa_{1},\ \left\|b\left(x\right)-b\left(y\right)\right\|_{2}\le \kappa_{1}\left(\left\|x-y\right\|_{2}^{\beta}\vee\left\|x-y\right\|_{2}\right).
    \end{align*}
\end{itemize}

Under ($H_{\alpha}^{a}$) and ($H_{\beta}^{b}$), the SDE has a unique weak solution \citep[see][]{SV79,RW00,BP09,DM10,Men11}.

Let $\rho$ be a nonnegative smooth function with support in the unit ball of $\left(\R^{d},\left\|\cdot\right\|_{2}\right)$ and $\int_{\R^{d}}\rho\left(x\right)\diff x=1$.
Define $\rho_{\epsilon}\left(x\right):=\epsilon^{-d}\rho\left(\epsilon^{-1}x\right)$ for $\epsilon\in\left(0,1\right]$ and $b_{\epsilon}\left(x\right):=b\ast \rho_{\epsilon}\left(x\right)=\int_{\R^{d}}b\left(y\right)\rho_{\epsilon}\left(x-y\right)\diff y$.
The following then holds:
\begin{align}
    \left\|\left\|\nabla_{x} b_{1}\right\|_{2}\right\|_{\infty}:=\sup_{x\in\R^{d}}\left\|\nabla_{x}^{n}b_{1}\left(x\right)\right\|_{2}&\le \kappa_{1}\vol\left(B_{1}\left(\zero\right)\right)\sup_{x:\left\|x\right\|_{2}\le 1}\left\|\nabla_{x}^{n}\rho\left(x\right) \right\|_{2}
\end{align}
\citep[see (1.9) of][]{MPZ21}.
% For simplicity of the discussion, we restrict $\beta\in\left(0,1\right]$ whilst the original study of \citet{MPZ21} yields for $\beta\in\left[0,1\right]$ with regularized flows.
Let $\varphi_{t}^{\left(\epsilon\right)}\left(x\right)$, $t\ge0$ be a deterministic flow $\dot{\varphi}_{t}^{\left(\epsilon\right)}\left(x\right):=b_{\epsilon}(\varphi_{t}^{(\epsilon)}(x))$, $\varphi_{0}^{(\epsilon)}\left(x\right)=x$.

The following Aronson-type estimates for the transition density function of $X_{t}$ hold.

\begin{theorem}[a corollary of Theorem 1.2 by \citealp{MPZ21}]\label{thm:pr:Aronson}
Under ($H_{\alpha}^{a}$) and ($H_{\beta}^{b}$), for any $T>0$, $t\in\left(0,T\right)$ and $x\in\R^{d}$, the unique weak solution $X_{t}^{a,b}\left(x\right)$ admits a density $p_{t}^{a,b}\left(x,y\right)$, which is continuous in $x,y\in\R^{d}$.
Moreover, $p_{t}^{a,b}$ has the following properties:
\begin{enumerate}
    \item[(i)] (Two-sided density bounds) there exist constants $\lambda_{0}\in\left(0,1\right]$ and $C_{0}\ge 1$ depending only on $\left(T,\alpha,\beta,\kappa_{0},\kappa_{1},d\right)$ such that for all $t\in\left(0,T\right)$ and $x,y\in\R^{d}$,
    \begin{align*}
        \frac{1}{C_{0}t^{d/2}}\exp\left(-\frac{\left\|y-\varphi_{t}^{\left(1\right)}\left(x\right)\right\|_{2}^{2}}{\lambda_{0}t}\right)\le p_{t}^{a,b}\left(x,y\right)\le \frac{C_{0}}{t^{d/2}}\exp\left(-\frac{\lambda_{0}\left\|y-\varphi_{t}^{\left(1\right)}\left(x\right)\right\|_{2}^{2}}{t}\right);
    \end{align*}
    \item[(ii)] (Gradient estimate in $x$) there exist constants $\lambda_{1}\in\left(0,1\right]$ and $C_{1}\ge 1$ depending only on $\left(T,\alpha,\beta,\kappa_{0},\kappa_{1},d\right)$ such that for all $t\in\left(0,T\right)$ and $x,y\in\R^{d}$,
    \begin{align*}
        \left\|\nabla_{x}p_{t}^{a,b}\left(x,y\right)\right\|_{2}\le \frac{C_{1}}{t^{\left(d+1\right)/2}}\exp\left(-\frac{\lambda_{1}\left\|y-\varphi_{t}^{\left(1\right)}\left(x\right)\right\|_{2}^{2}}{t}\right).
    \end{align*}
\end{enumerate}
\end{theorem}

$C_{j}$ and $\lambda_{j}$ are completely determined by $\left(T,\alpha,\beta,\kappa_{0},\kappa_{1},d\right)$; hence, for SDEs satisfying ($H_{\alpha}^{a}$) and $H_{\beta}^{b}$ for the same parameters, the density estimates are uniform across those models.

\begin{lemma}\label{lem:pr:mollifier}
    Under ($H_{\beta}^{b}$), the following holds:
    \begin{align*}
        \left\|\varphi_{t}^{\left(1\right)}\left(x\right)-x\right\|_{2}\le \kappa_{1}t\left(2+\left\|x\right\|_{2}^{\beta}\vee\left\|x\right\|_{2}\right)\exp\left(\left\|\left\|\nabla_{x} b_{1}\right\|_{2}\right\|_{\infty}t\right).
    \end{align*}
\end{lemma}

\begin{proof}
A discussion similar to \citet{MPZ21} yields
\begin{align*}
    \left\|\varphi_{t}^{\left(1\right)}\left(x\right)-x\right\|_{2}%&= \left\|\int_{0}^{t}b\left(\varphi_{s}\left(x\right)\right)\diff s\right\|\\
    &\le \int_{0}^{t}\left\|b_{1}\left(\varphi_{s}^{\left(1\right)}\left(x\right)\right)\right\|_{2}\diff s\\
    &\le \int_{0}^{t}\left\|b_{1}\left(\varphi_{s}^{\left(1\right)}\left(x\right)\right)-b_{1}\left(x\right)\right\|_{2}\diff s+\int_{0}^{t}\left\|b_{1}\left(x\right)\right\|_{2}\diff s\\
    &\le\int_{0}^{t}\left\|\int_{0}^{1}\left(\nabla_{x} b_{1}\right)\left(x+u\left(\varphi_{s}^{\left(1\right)}\left(x\right)-x\right)\right)\left(\varphi_{s}^{\left(1\right)}\left(x\right)-x\right)\diff u\right\|_{2}\diff s\\
    &\quad+\int_{0}^{t}\left\|b_{1}\left(x\right)-b\left(x\right)\right\|_{2}\diff s+\int_{0}^{t}\left\|b\left(x\right)-b\left(\zero\right)\right\|_{2}\diff s+\int_{0}^{t}\left\|b\left(\zero\right)\right\|_{2}\diff s\\
    &\le \left\|\left\|\nabla_{x} b_{1}\right\|_{2}\right\|_{\infty}\int_{0}^{t}\left\|\varphi_{s}^{\left(1\right)}\left(x\right)-x\right\|_{2}\diff s+\kappa_{1}t+\kappa_{1}\left(\left\|x\right\|_{2}^{\beta}\vee\left\|x\right\|_{2}\right)t+\kappa_{1}t\\
    &\le \kappa_{1}\left(2+\left\|x\right\|_{2}^{\beta}\vee\left\|x\right\|_{2}\right)t+\left\|\left\|\nabla_{x} b_{1}\right\|_{2}\right\|_{\infty}\int_{0}^{t}\left\|\varphi_{s}^{\left(1\right)}\left(x\right)-x\right\|_{2}\diff s.
\end{align*}
Hence, Gronwall's inequality yields the statement.
\end{proof}

We verify the local Dobrushin condition using Theorem \ref{thm:pr:Aronson} and Lemma \ref{lem:pr:mollifier}.
Let us omit the explicit dependence of constants on $\rho$ in the statements because we consider a fixed $\rho$.
% We obtain the following local Dobrushin condition dependent only on $\left(T,\alpha,\beta,\kappa_{0},\kappa_{1},d\right)$.

\begin{proposition}\label{prp:pr:Dobrushin}
    For fixed $T_{1}, T_{2}>0$ with $T_{1}<T_{2}$ and compact and convex $K\subset \R^{d}$, there exists a constant $\delta>0$ dependent only on $\left(T_{1},T_{2},\alpha,\beta,\kappa_{0},\kappa_{1},d,K\right)$ such that for all $t\in \left(T_{1},T_{2}\right)$, 
    \begin{align*}
        \sup_{x,y\in K}\left\|P_{2t}^{a,b}\left(x,\cdot\right)-P_{2t}^{a,b}\left(y,\cdot\right)\right\|_{\TV}\le 2-\delta.
    \end{align*}
\end{proposition}

\begin{proof}
The proof is similar to that for Propositions 2.9.1 and 2.9.3 by \citet{Kul17}; however, we present the proof to illustrate the construction of $\delta$.

\textbf{(Step 1)}
Owing to Theorem \ref{thm:pr:Aronson}-(ii), the evaluation of the total variation of transition kernels is straightforward.
For all $x,x'\in\R^{d}$,
\begin{align*}
    \left\|P_{t}^{a,b}\left(x,\cdot\right)-P_{t}^{a,b}\left(x',\cdot\right)\right\|_{\TV}
    &= \int_{\R^{d}}\left|p_{t}^{a,b}\left(x,y\right)-p_{t}^{a,b}\left(x',y\right)\right|\diff y.
\end{align*}
Theorem \ref{thm:pr:Aronson}-(ii), Lemma \ref{lem:pr:mollifier}, the convexity of $K$, the fact $\|v_{1}+v_{2}+v_{3}\|_{2}^{2}\le 3(\|v_{1}\|_{2}^{2}+\|v_{2}\|_{2}^{2}+\|v_{3}\|_{2}^{2}),v_{i}\in\R^{d}$, and the mean value theorem yield % the fact $\|v_{1}+v_{2}+v_{3}\|_{2}^{2}\le 3\left(\|v_{1}\|_{2}^{2}+\|v_{2}\|_{2}^{2}+\|v_{3}\|_{2}^{2}\right),v_{i}\in\R^{d}$ yield
\begin{align*}
    &\int_{\R^{d}}\left|p_{t}^{a,b}\left(x,y\right)-p_{t}^{a,b}\left(x',y\right)\right|\diff y\\
    &\le \left\|x-x'\right\|_{2}\int_{\R^{d}}\left(\max_{x\in K}\left\|\nabla_{x}p_{t}^{a,b}\left(x,y\right)\right\|_{2}\right)\diff y\\
    &\le \frac{C_{1}\left\|x-x'\right\|_{2}}{t^{\left(d+1\right)/2}}\int_{\R^{d}}\max_{x\in K}\exp\left(-\frac{\lambda_{1}\left\|y-\varphi_{t}^{\left(1\right)}\left(x\right)\right\|_{2}^{2}}{t}\right)\diff y\\
    &\le \frac{C_{1}\left\|x-x'\right\|_{2}}{t^{\left(d+1\right)/2}}\int_{\R^{d}}\max_{x\in K}\exp\left(-\frac{\lambda_{1}\left(\left\|y\right\|_{2}^{2}/3-\left\|\varphi_{t}^{\left(1\right)}\left(x\right)-x\right\|_{2}^{2}-\left\|x\right\|_{2}^{2}\right)}{t}\right)\diff y\\
    &\le \left\|x-x'\right\|_{2}D_{1},
\end{align*}
where $D_{1}>0$ is a positive constant dependent only on $\left(T_{1},T_{2},\alpha,\beta,\kappa_{0},\kappa_{1},d,\rho,K\right)$.
Hence, for all $\delta_{1}\in\left(0,2/D_{1}\right)$, $x,x'\in K$ with $\left\|x-x'\right\|_{2}\le \delta_{1}$,
\begin{align*}
    \sup_{x,x'\in K:\left\|x-x'\right\|_{2}\le \delta_{1}}\left\|P_{t}^{a,b}\left(x,\cdot\right)-P_{t}^{a,b}\left(x',\cdot\right)\right\|_{\TV}\le D_{1}\delta_{1} <2.
\end{align*}

\textbf{(Step 2)}
By Theorem \ref{thm:pr:Aronson}-(i) and Lemma \ref{lem:pr:mollifier}, for any $x\in K$ and $\delta_{2}>0$, 
\begin{align*}
    P_{t}^{a,b}\left(x,B_{\delta_{2}}\left(\zero\right)\right)&=\int_{B_{\delta_{2}}\left(\zero\right)}p_{t}^{a,b}\left(x,y\right)\diff y\\
    &\ge \int_{B_{\delta_{2}}\left(\zero\right)}\frac{1}{C_{0}t^{d/2}}\exp\left(-\frac{\left\|y-\varphi_{t}^{\left(1\right)}\left(x\right)\right\|_{2}^{2}}{\lambda_{0}t}\right)\diff y\\
    &\ge \int_{B_{\delta_{2}}\left(\zero\right)}\frac{1}{C_{0}t^{d/2}}\exp\left(-\frac{2\left\|y\right\|_{2}^{2}+2\left\|\varphi_{t}^{\left(1\right)}\left(x\right)\right\|_{2}^{2}}{\lambda_{0}t}\right)\diff y\\
    &\ge \int_{B_{\delta_{2}}\left(\zero\right)}\frac{1}{C_{0}t^{d/2}}\exp\left(-\frac{2\delta_{2}^{2}+2\left\|\varphi_{t}^{\left(1\right)}\left(x\right)\right\|_{2}^{2}}{\lambda_{0}t}\right)\diff y\\
    %&=\frac{\vol\left(B_{1}\left(\zero\right)\right)\delta_{2}^{d}}{C_{0}t^{d/2}}\exp\left(-\frac{2\delta_{2}^{2}+2\left\|\varphi_{t}^{\left(1\right)}\left(x\right)\right\|_{2}^{2}}{\lambda_{0}t}\right)\\
    &\ge D_{2},
\end{align*}
where $D_{2}>0$ is a positive constant dependent only on $\left(T_{1},T_{2},\alpha,\beta,\kappa_{0},\kappa_{1},d,\rho,K,\delta_{2}\right)$.
Let us restrict $\delta_{2}\in\left(0,2/D_{1}\right)$ satisfying $B_{\delta_{2}}\left(\zero\right)\subset K$.
We then have
\begin{align*}
    \left\|P_{2t}^{a,b}\left(x,\cdot\right)-P_{2t}^{a,b}\left(x',\cdot\right)\right\|_{\TV}&=\left\|\int_{\R^{2d}} \left(P_{t}^{a,b}\left(x,\diff y\right)P_{t}^{a,b}\left(y,\cdot\right)-P_{t}^{a,b}\left(x',\diff y'\right)P_{t}^{a,b}\left(y',\cdot\right)\right)\right\|_{\TV}\\
    &=\left\|\int_{\R^{2d}} \left(P_{t}^{a,b}\left(y,\cdot\right)-P_{t}^{a,b}\left(y',\cdot\right)\right)P_{t}^{a,b}\left(x,\diff y\right)P_{t}^{a,b}\left(x',\diff y'\right)\right\|_{\TV}\\
    &\le \int_{\R^{2d}} \left\|P_{t}^{a,b}\left(y,\cdot\right)-P_{t}^{a,b}\left(y',\cdot\right)\right\|_{\TV}P_{t}^{a,b}\left(x,\diff y\right)P_{t}^{a,b}\left(x',\diff y'\right)\\
    &\le 2\int_{\R^{2d}\backslash \left(B_{\delta_{2}}\left(\zero\right)\times B_{\delta_{2}}\left(\zero\right)\right)} P_{t}^{a,b}\left(x,\diff y\right)P_{t}^{a,b}\left(x',\diff y'\right)\\
    &\quad+D_{1}\delta_{2}\int_{B_{\delta_{2}}\left(\zero\right)\times B_{\delta_{2}}\left(\zero\right)}P_{t}^{a,b}\left(x,\diff y\right)P_{t}^{a,b}\left(x',\diff y'\right)\\
    &=2\left(1-P_{t}^{a,b}\left(x,B_{\delta_{2}}\left(\zero\right)\right)P_{t}^{a,b}\left(x',B_{\delta_{2}}\left(\zero\right)\right)\right)\\ % reduction
    &\quad+D_{1}\delta_{2}P_{t}^{a,b}\left(x,B_{\delta_{2}}\left(\zero\right)\right)P_{t}^{a,b}\left(x',B_{\delta_{2}}\left(\zero\right)\right)\\ % reduction
    &=2-\left(2-D_{1}\delta_{2}\right)P_{t}^{a,b}\left(x,B_{\delta_{2}}\left(\zero\right)\right)P_{t}^{a,b}\left(x',B_{\delta_{2}}\left(\zero\right)\right)\\ % reduction
    &\le 2-\left(2-D_{1}\delta_{2}\right)D_{2}^{2}.
\end{align*}
Hence, $\delta:=\left(2-D_{1}\delta_{2}\right)D_{2}^{2}>0$ leads to the statement.
\end{proof}

\subsection{Lyapunov-type condition}
In addition to ($H_{\beta}^{b}$) and ($H_{\alpha}^{a}$), we also set the following drift condition for exponential ergodicity:
\begin{enumerate}
    \item[($L_{\gamma}^{b}$)] There exist constants $\gamma\ge0$ and $\varkappa_{1}>0$ such that for all $x\in\R^{d}$,
    \begin{align*}
        \left\langle b\left(x\right),x\right\rangle \le  -\varkappa_{1}^{-1}\left\|x\right\|_{2}^{1+\gamma}+\varkappa_{1}.
    \end{align*}
\end{enumerate}

We define the operator $\mathcal{L}^{a,b}$ such that for all $f\in\mathcal{C}^{2}\left(\R^{d}\right)$,
\begin{align}
    \mathcal{L}^{a,b}f\left(x\right):=\left\langle b\left(x\right),\partial_{x}f\left(x\right)\right\rangle +\frac{1}{2}\tr\left(a^{\otimes2}\left(x\right)\partial_{x}^{2}f\left(x\right)\right).
\end{align}
Let $\E_{x}^{a,b}$ denote the expectation with respect to the weak solution for fixed $a,b$, and $x$.

\begin{proposition}\label{prp:pr:Lyapunov}
    Under ($H_{\alpha}^{a}$), ($H_{\beta}^{b}$), and ($L_{\gamma}^{b}$), for all $\left(\gamma,\nu\right)\in\R_{+}^{2}$ such that $\gamma=0$ and $\nu\in\left(0,2\varkappa_{1}^{-1}/\kappa_{0}\right)$ or arbitrary $\gamma>0$ and $\nu>0$, there exist positive constants $E_{1},E_{2}>0$ dependent only on $\left(\gamma,\nu,\kappa_{0},\varkappa_{1},d\right)$ such that for any $h>0$ and $x\in\R^{d}$,
    \begin{align*}
        \E_{x}^{a,b}\left[V\left(X_{h}^{a,b}\right)\right]-V\left(x\right)\le -\left(1-e^{-E_{1}h}\right)V\left(x\right)+\frac{E_{2}\left(1-e^{-E_{1}h}\right)}{E_{1}},
    \end{align*}
    where $V:=\exp\left(\nu\sqrt{1+\left\|x\right\|_{2}^{2}}\right)$.
\end{proposition}

\begin{proof}
$V\left(x\right)$ has the following properties:
\begin{align*}
    \partial_{x}V\left(x\right)=\frac{\nu V\left(x\right)}{\sqrt{1+\left\|x\right\|_{2}^{2}}}x,\ \partial_{x}^{2}V\left(x\right)=\frac{\nu^{2} V\left(x\right)}{1+\left\|x\right\|_{2}^{2}}xx^{\top}-\frac{\nu V\left(x\right)}{\left(1+\left\|x\right\|_{2}^{2}\right)^{3/2}}xx^{\top}+\frac{\nu V\left(x\right)}{\sqrt{1+\left\|x\right\|_{2}^{2}}}I_{d},
\end{align*}
and hence,
\begin{align*}
    \generator^{a,b}V&=\left\langle b\left(x\right), \partial_{x}V\left(x\right)\right\rangle + \frac{1}{2}\tr\left(a^{\otimes2}\left(x\right)\partial_{x}^{2}V\left(x\right)\right)\\
    &\le \frac{\nu V\left(x\right)}{\sqrt{1+\left\|x\right\|_{2}^{2}}}\left\langle b\left(x\right), x\right\rangle+\frac{\nu V\left(x\right)}{2}\left(\nu\frac{x^{\top}a^{\otimes2}\left(x\right)x}{1+\left\|x\right\|_{2}^{2}}+\frac{x^{\top}a^{\otimes2}\left(x\right)x}{\left(1+\left\|x\right\|_{2}^{2}\right)^{3/2}}+\frac{\tr a^{\otimes2}\left(x\right)}{\sqrt{1+\left\|x\right\|_{2}^{2}}}\right)\\
    &\le V\left(x\right)\left[\frac{\nu \left(-\varkappa_{1}^{-1}\left\|x\right\|_{2}^{1+\gamma}+\varkappa_{1}\right)}{\sqrt{1+\left\|x\right\|_{2}^{2}}}+\frac{\nu }{2}\left(\nu\kappa_{0}+\frac{\kappa_{0}\left(1+d\right)}{\sqrt{1+\left\|x\right\|_{2}^{2}}}\right)\right].
\end{align*}
For any $R>0$,  
\begin{align*}
    \sup_{x:\left\|x\right\|_{2}\le R}\generator^{a,b}V
    &\le \exp\left(\nu\sqrt{1+R^{2}}\right)\left(\nu\varkappa_{1}+\frac{\nu \kappa_{0}}{2}\left(\nu+1+d\right)\right).
\end{align*}
The assumption that $\gamma=0$ and $\nu\in\left(0,2\varkappa_{1}^{-1}/\kappa_{0}\right)$ or $\gamma>0$ and $\nu>0$ yields the existence of $R_{1}:=R_{1}\left(\gamma,\nu,\kappa_{0},\varkappa_{1},d\right)>0$ such that
\begin{align*}
    &\sup_{x:\left\|x\right\|_{2}\ge R_{1}}\left(\frac{\nu \left(-\varkappa_{1}^{-1}\left\|x\right\|_{2}^{1+\gamma}+\varkappa_{1}\right)}{\sqrt{1+\left\|x\right\|_{2}^{2}}}+\frac{\nu }{2}\left(\nu\kappa_{0}+\frac{\kappa_{0}\left(1+d\right)}{\sqrt{1+\left\|x\right\|_{2}^{2}}}\right)\right)\\
    &=\frac{\nu \left(-\varkappa_{1}^{-1}R_{1}^{1+\gamma}+\varkappa_{1}\right)}{\sqrt{1+R_{1}^{2}}}+\frac{\nu }{2}\left(\nu\kappa_{0}+\frac{\kappa_{0}\left(1+d\right)}{\sqrt{1+R_{1}^{2}}}\right)<0.
\end{align*}
Therefore, by fixing such $R_{1}$,
\begin{align*}
    \generator^{a,b}V\le -E_{1}V\left(x\right)+E_{2},
\end{align*}
where $E_{1},E_{2}>0$ are positive constants dependent only on $\left(\gamma,\nu,\kappa_{0},\varkappa_{1},d\right)$ such that
\begin{align*}
    E_{1}&:=-\frac{\nu \left(-\varkappa_{1}^{-1}R_{1}^{1+\gamma}+\varkappa_{1}\right)}{\sqrt{1+R_{1}^{2}}}-\frac{\nu }{2}\left(\nu\kappa_{0}+\frac{\kappa_{0}\left(1+d\right)}{\sqrt{1+R_{1}^{2}}}\right)\\
    E_{2}&:=\exp\left(\nu\sqrt{1+R_{1}^{2}}\right)\left(\nu \varkappa_{1}+\frac{\nu \kappa_{0}}{2}\left(\nu+1+d\right)\right).
\end{align*}
This inequality and Theorem 3.2.3 by \citet{Kul17} together prove the statement.
\end{proof}

The next corollary follows immediately.
\begin{corollary}\label{cor:pr:Moments}
    Under the same assumptions as Proposition \ref{prp:pr:Lyapunov}, we have
    \begin{align*}
        \sup_{t\ge0}\E_{x}^{a,b}\left[\exp\left(\nu\sqrt{1+\left\|X_{t}^{a,b}\right\|_{2}^{2}}\right)\right]&\le \exp\left(\nu\sqrt{1+\left\|x\right\|_{2}^{2}}\right)+\frac{E_{2}}{E_{1}}.
    \end{align*}
    For any $m\ge0$, we also have
    \begin{align*}
        \sup_{t\ge0}\E_{x}^{a,b}\left[\left\|X_{t}^{a,b}\right\|_{2}^{m}\right]\le \frac{m!}{\nu^{m}}\left(\exp\left(\nu\sqrt{1+\left\|x\right\|_{2}^{2}}\right)+\frac{E_{2}}{E_{1}}\right).
    \end{align*}
\end{corollary}

\subsection{Harris-type theorem}

The following exponential ergodicity with uniform constants is an immediate consequence of Theorems 2.6.1 and 2.6.3 and Corollary 2.8.3 by \citet{Kul17}.
\begin{theorem}\label{thm:pr:Harris}
Under the assumptions ($H_{\alpha}^{a}$), ($H_{\beta}^{b}$), and ($L_{\gamma}^{b}$) with $\gamma\ge0$, there exists a unique invariant probability measure $\Pi^{a,b}$ such that for all $t\ge 0$ and $x\in\R^{d}$,
\begin{align*}
    \left\|P_{t}^{a,b}\left(x,\cdot\right)-\Pi^{a,b}\left(\cdot\right)\right\|_{\TV}\le c_{1}\exp\left(-t/c_{2} \right)\left(V\left(x\right)+c_{3}\right),
\end{align*}
where $V\left(x\right):=\exp\left(\nu\sqrt{1+\left\|x\right\|_{2}^{2}}\right)$, and $c_{1},c_{2},c_{3},\nu>0$ are positive constants dependent only on $\left(\alpha,\beta,\gamma,\kappa_{0},\kappa_{1},\varkappa_{1},d\right)$.
\end{theorem}
\begin{proof}
It is sufficient to see that the statement holds with $c_{3}=E_{2}/E_{1}$.
From Theorems 2.6.1 and 2.6.3 and Corollary 2.8.3 by \citet{Kul17}, Proposition \ref{prp:pr:Dobrushin} with fixed $T_{1}=1$ and $T_{2}=2$ (this choice is arbitrary) and Proposition \ref{prp:pr:Lyapunov} yield
\begin{align*}
    \left\|P_{t}^{a,b}\left(x,\cdot\right)-\Pi^{a,b}\left(\cdot\right)\right\|_{\TV}\le c_{1}\exp\left(-t/c_{2} \right)\left(V\left(x\right)+\int V\left(y\right)\Pi^{a,b}\left(\diff y\right)\right),
\end{align*}
where $c_{1},c_{2},\nu>0$ are positive constants dependent only on $\left(\alpha,\beta,\gamma,\kappa_{0},\kappa_{1},\varkappa_{1},d\right)$.
It holds that
\begin{align*}
    \int_{\R^{d}}\exp\left(\nu\sqrt{1+\left\|y\right\|_{2}^{2}}\right)\Pi^{a,b}\left(\diff y\right)&\le \liminf_{t\to\infty}\E_{x}^{a,b}\left[\exp\left(\nu\sqrt{1+\left\|X_{t}^{a,b}\right\|_{2}^{2}}\right)\right]\le \frac{E_{2}}{E_{1}}
\end{align*}
by $P_{t}^{a,b}\to\Pi^{a,b}$ in total variation for all $x\in\R^{d}$, Fatou's lemma, Skorohod's representation theorem, and Proposition \ref{prp:pr:Lyapunov}.
\end{proof}

Theorem \ref{thm:pr:Harris} leads to the simultaneous exponential ergodicity of the $d$-dimensional diffusion processes defined by the SDEs satisfying ($H_{\alpha}^{a}$), ($H_{\beta}^{b}$), and ($L_{\gamma}^{b}$) with the same constants $\left(\alpha,\beta,\gamma,\kappa_{0},\kappa_{1},\varkappa_{1}\right)$.

\section{Estimation of stochastic differential equations}

% \subsection{Latent models}
We consider the estimation of the unknown drift coefficient $b:\R^{d}\to\R^{d}$ and diffusion coefficient $a:\R^{d}\to\R^{d}\otimes \R^{d}$ of the following SDE based on discrete observations of $X_{t}$:
\begin{align}
    \diff X_{t}^{a,b}=b\left(X_{t}^{a,b}\right)\diff t + a\left(X_{t}^{a,b}\right)\diff w_{t},\ X_{0}=x,\ t\ge0,\label{eq:SDE}
\end{align}
where $x\in\R^{d}$ is a deterministic initial value and $w_{t}$ is a $d$-dimensional Wiener process.
%, $a:\R^{d}\to\R^{d}\otimes\R^{d}$ is the diffusion coefficient that we may not know.
We do not necessarily aim to estimate the optimal parameters by considering $a$ and $b$ to be included in the statistical model; rather, we consider misspecified modelling and estimate the quasi-optimal parameter \citep{UY11} to know the model closest to $a$ and $b$ in the sense of the $L^{2}$-distances with respect to invariant probability measures.

We apply the discussion in Section 2 to present model-wise risk bounds for parametric estimation via online subgradient descent, which is obtained by setting $\psi(\cdot)=\|\cdot\|_{2}^{2}/2$ and $\left\|\cdot\right\|=\left\|\cdot\right\|_{\ast}=\left\|\cdot\right\|_{2}$, and that in Section 3 to render those upper bounds uniform with respect to SDEs in certain classes.
The notation for the classes of coefficients are as follows: let $\alpha,\gamma,\kappa_{0},\kappa_{1},\varkappa_{1}>0$, $\beta\ge 0$,   $\varpi:=\left(\alpha,\beta,\gamma,\kappa_{0},\kappa_{1},\varkappa_{1}\right)$, and $S_{\varpi}$ be the class of coefficients such that
\begin{align}
    S_{\varpi}:=\left\{\begin{array}{l|l}
    \multirow{2}{*}{$\left(a,b\right)$}& a\text{ satisfies }(H_{\alpha}^{a})\text{ and }b\text{ satisfies }(H_{\beta}^{b})\text{ and }(L_{\gamma}^{b})\\
    &\text{with the same constants }\alpha,\beta,\gamma,\kappa_{0},\kappa_{1},\varkappa_{1}\end{array}\right\}.
\end{align}
We finally obtain the risk bounds that uniformly hold for all $\left(a,b\right)\in S_{\varpi}$ with fixed $\varpi$ by combining Theorems \ref{thm:op:SMDEx} and \ref{thm:pr:Harris} and Corollary \ref{cor:pr:Moments}.

% Let us use the notation such that $\Delta_{i}X=X_{ih_{n}}^{a,b}-X_{\left(i-1\right)h_{n}}^{a,b}$ for all $i=1,\ldots,n$.

Our estimation is based on discrete observations $\{X_{ih_{n}}^{a,b}\}_{i=0,\ldots,n}$ for a sample size of $n\in\N$ and the discretization step $h_{n}\in\left(0,1\right]$.
% For abbreviation, we use the notation $\Delta_{i}X=X_{ih_{n}}^{a,b}-X_{\left(i-1\right)h_{n}}^{a,b}$ for all $i=1,\ldots,n$
We use the notation $\Delta_{i}X=X_{ih_{n}}^{a,b}-X_{\left(i-1\right)h_{n}}^{a,b}$ for all $i=1,\ldots,n$ and $\F^{a,b}=\{\calF_{i}^{a,b};i\in\N_{0}\}$, where $\calF_{i}^{a,b}=\sigma(X_{t}^{a,b};t\le ih_{n})$.
% In addition, we write the weak solution of \eqref{eq:SDE} for each $\left(a,b\right)\in S_{\varpi}$ and the corresponding expectation $\E_{x}^{a,b}$ as $(\{X_{t}\}_{t\ge0},\{w_{t}\}_{t\ge0})$, $(\Omega,\mathcal{A},P)$, and $\{\mathcal{A}_{t}\}_{t\ge0}$ and $\E$ in cases wherein dependence on $(a,b)$ does not matter and no confusion can arise.
% why we do not need $\mathcal{A}$; K's are only defined with the $\calF_{i}$-measurability.
In addition, we write $\E_{x}^{a,b}$, $X_{t}^{a,b}$, $\F^{a,b}$, and $\calF_{i}^{a,b}$ simply as $\E$, $X_{t}$, $\F$, and $\calF_{i}$ in cases wherein no confusion can arise.

\subsection{Estimation of drift coefficients}
In the first place, we consider the estimation of drift coefficients.
\subsubsection{Statistical modelling}
Let $\Xi=\R^{d}$, $\Theta\in\mathcal{B}\left(\R^{p}\right)$ be the compact convex parameter space, $N_{\Theta}$ be an open neighbourhood of $\Theta$, and $R:=\sup\left\{\left\|\theta-\theta'\right\|_{2};\theta,\theta'\in\Theta\right\}$.
% The loss function on $\Theta$ are defined with unknown $b$ and a known triple $\left(\bModel,M,J\right)$: (1) the Borel-measurable parametric model $\bModel:\R^{d}\times N_{\Theta}\to\R^{d}$; (2) the weight function $M:\R^{d}\to\R^{d}\times\R^{d}$ such that $M\left(\xi\right)$ is $\mathcal{B}\left(\R^{d}\right)$-measurable and positive semi-definite for all $\xi\in\R^{d}$; and (3) the regularization term $J:N_{\Theta}\to\R$ is a measurable function.
The loss function on $\Theta$ is defined with unknown $b$ and a known triple $\left(\bModel,M,J\right)$ of functions with Borel-measurable elements : (1) the parametric model $\bModel:\R^{d}\times N_{\Theta}\to\R^{d}$; (2) the weight function $M:\R^{d}\to\R^{d}\otimes\R^{d}$, which is positive semi-definite for all $\xi\in\R^{d}$; and (3) the regularization term $J:N_{\Theta}\to\R$.

We define a function $F:N_{\Theta}\times \R^{d}\to \R$ such that
\begin{align}
    F\left(\theta;\xi\right)=F^{b}\left(\theta;\xi\right):=\frac{1}{2}M\left(\xi\right)\left[\left(\bModel\left(\xi,\theta\right)-b\left(\xi\right)\right)^{\otimes2}\right]+J\left(\theta\right).
\end{align}
%and assume that $F$ is $\left(\mathcal{B}\left(\R^{p}\right)|_{N_{\Theta}}\right)\otimes\mathcal{B}\left(\R^{d}\right)$-measurable.
We consider the minimization problem of the following loss function on $\Theta$:
\begin{align}
    f^{a,b}\left(\theta\right):=\int F^{b}\left(\theta;\xi\right)\Pi^{a,b}\left(\diff \xi\right),
\end{align}
where $\Pi^{a,b}$ is the invariant probability measure of $X_{t}^{a,b}$.

Clearly, $F\left(\theta;x\right)$, which depends on the unknown coefficient $b$, is unknown.
Hence, we consider the approximated loss functions based on discrete observations and observe the performance of the estimator given by online gradient descents.

The sampled loss functions are given by the $h_{n}$-skeleton of $X_{t}^{a,b}$:
\begin{align}
    F\left(\theta;\xi_{i}\right)=\frac{1}{2}M\left(\xi_{i}\right)\left[\left(\bModel\left(\xi_{i},\theta\right)-b\left(\xi_{i}\right)\right)^{\otimes2}\right]+J\left(\theta\right),
\end{align}
where $\xi_{i}=X_{\left(i-1\right)h_{n}}^{a,b}$.
$H_{i,n}\left(\theta\right)$, a random function on $\Theta$, should be sufficiently close to $F$; hence, we set
\begin{align}
    H_{i,n}\left(\theta\right)&:=\frac{1}{2h_{n}^{2}}M\left(X_{\left(i-1\right)h_{n}}\right)\left[\left(\Delta_{i}X-h_{n}\bModel\left(X_{\left(i-1\right)h_{n}},\theta\right)\right)^{\otimes2}-\left(\Delta_{i}X-h_{n}b\left(X_{\left(i-1\right)h_{n}}\right)\right)^{\otimes2}\right]\notag\\
    &\qquad+J\left(\theta\right).
    %&=\frac{1}{2h^{2}}M_{i-1}\left[\left(\Delta_{i}X-hb_{i-1}+hb_{i-1}-h\bModel_{i-1}\left(\theta\right)\right)^{\otimes2}-\left(\Delta_{i}X-hb_{i-1}\right)^{\otimes2}\right]+J\left(\theta\right)\\
    %&=\frac{1}{h}M_{i-1}\left[b_{i-1}-\bModel_{i-1}\left(\theta\right),\Delta_{i}X-hb_{i-1}\right]+F\left(\theta;\xi_{i}\right).
\end{align}
A simple computation leads to the equality
\begin{align*}
    &H_{i,n}\left(\theta\right)-F\left(\theta;\xi_{i}\right)\\
    &=\frac{1}{h_{n}}M\left(X_{\left(i-1\right)h_{n}}\right)\left[b\left(X_{\left(i-1\right)h_{n}}\right)-\bModel\left(X_{\left(i-1\right)h_{n}},\theta\right),\Delta_{i}X-h_{n}b\left(X_{\left(i-1\right)h_{n}}\right)\right].
\end{align*}

We set the following assumptions on $F$, $H_{i,n}$, and $\left(\bModel,M,J\right)$:
\begin{itemize}
    \item[(D1)] For all $\xi,\xi'\in\R^{d}$, 
    \begin{align*}
        \frac{1}{2}M\left(\xi\right)\left[\left(\bModel\left(\xi,\theta\right)-\xi'\right)^{\otimes2}\right]+J\left(\theta\right)
    \end{align*}
    is convex with respect to $\theta\in N_{\Theta}$.
     % Moreover, there exist positive constants $G>0$ and $K>0$ such that for all $n\in\N$ and $\calF_{i-1}$-measurable $\Theta$-valued $\theta$,
     Moreover, there exist positive constants $G>0$ and $\check{K}>0$, a $\mathcal{B}(\R^{d})\otimes\left(\mathcal{B}\left(\R^{p}\right)|_{N_{\Theta}}\right)$-measurable function $\sfG$, and an $\calF_{i}^{a,b}\otimes\left(\mathcal{B}\left(\R^{p}\right)|_{N_{\Theta}}\right)$-measurable random function $\sfK$ such that $\sfG\left(\theta;\xi\right)\in \partial F\left(\theta;\xi\right)$ and $\sfK_{i,n}\left(\theta\right)\in \partial H_{i,n}\left(\theta\right)$ a.s.\ for all $\xi\in\R^{d}$, $\theta\in \Theta$, and $i=1,\ldots,n$, and 
    \begin{align*}
        \sup_{\left(a,b\right)\in S_{\varpi}}\sup_{i\in\N}\E_{x}^{a,b}\left[\left\|\sfG\left(\vartheta_{i};\xi_{i}\right)\right\|_{2}^{2}\right]\le G^{2},\ 
        \sup_{\left(a,b\right)\in S_{\varpi}}\sup_{i=1,\ldots,n}\E_{x}^{a,b}\left[\left\|\sfK_{i,n}\left(\vartheta_{i}\right)\right\|_{2}^{2}\right]\le \frac{\check{K}^{2}}{h_{n}}
    \end{align*}
    for all $n\in\N$ and sequence of $\calF_{i\wedge n-1}^{a,b}$-measurable $\Theta$-valued random variables $\vartheta_{i}$.
    % where $\sfG$ is a $\mathcal{B}(\R^{d})\otimes\left(\mathcal{B}\left(\R^{p}\right)|_{N_{\Theta}}\right)$-measurable function such that $\sfG\left(\theta;\xi_{i}\right)\in \partial F\left(\theta;\xi_{i}\right)$, and $\sfK$ is an $\calF_{i}\otimes\left(\mathcal{B}\left(\R^{p}\right)|_{N_{\Theta}}\right)$-measurable random function such that $\sfK_{i,n}\left(\theta\right)\in \partial H_{i,n}\left(\theta\right)$ a.s.\ for all $\theta\in\Theta$ (these subdifferentials are well-defined).
    \item[(D2)] There exists a constant $\zeta>0$ such that for all $x\in\R^{d}$ and $\theta\in N_{\Theta}$,
    \begin{align*}
        \left\|\bModel\left(x,\theta\right)\right\|_{2}\le \zeta\left(1+\left\|x\right\|_{2}^{\zeta}\right),\ 
        \left\|M\left(x\right)\right\|_{2}\le \zeta.
    \end{align*}
\end{itemize}
Under (D1), $\theta_{i+1}$ defined by the update rule \eqref{SMDUpdate} as well as $\xi_{i}=X_{\left(i-1\right)h_{n}}$ is $\calF_{i}$-measurable.
(D2) and Corollary \ref{cor:pr:Moments} yield that $f^{a,b}\left(\theta\right)<\infty$ for all $\theta\in N_{\Theta}$.

% We only consider the least-square-type objective functions and the contrast ones \citep[for example, see][]{Mas05} for SDEs with linear coefficients with respect to the parameters.
% The discussion on the other pairs of functions and coefficients runs in parallel.

% We consider the parametric estimation of a $d$-dimensional diffusion process $\left\{X_{t}\right\}_{t\ge0}$ defined by the following SDE:
% \begin{align}
%     \diff X_{t}=b\left(X_{t},\theta\right)\diff t+ a\left(X_{t}\right)\diff w_{t},\ X_{0}=x_{0},
% \end{align}
% where $\left\{w_{t}\right\}_{t\ge0}$ is a $r$-dimensional Wiener process, $x_{0}$ is a $d$-dimensional random variable independent of the Wiener process, $\theta\in\Theta$ is the unknown parameter of interest, $\Theta\subset\mathbf{R}^{p}$ is a bounded, open, and convex set admitting Sobolev's inequalities for embedding $W^{1,p}\left(\Theta\right)\hookrightarrow C\left(\Bar{\Theta}\right)$, and $a:\mathbf{R}^{d}\to\mathbf{R}^{d}\otimes\mathbf{R}^{r}$ and $b:\mathbf{R}^{d}\times\Theta\to\mathbf{R}^{d}$ are the known functions.
% Furthermore, let $\theta^{\star}\in\Theta$ denote the true value of $\theta$.

\subsubsection{Main results} 
For notational simplicity, we sometimes let $f_{i}:=f\left(X_{ih_{n}}\right)$ and $g_{i}\left(\theta\right):=g\left(X_{ih_{n}},\theta\right)$ for all $i=0,\ldots,n$, $\mathcal{B}(\R^{d})$-measurable $f$ and $\mathcal{B}(\R^{d})\otimes\left(\mathcal{B}\left(\R^{p}\right)|_{N_{\Theta}}\right)$-measurable $g$.

The following proposition provides the bound for the residual terms $\resid_{\tau,n}$.

\begin{proposition}\label{prp:st:b:Resid}
    Assume that (D2) holds.
    There exists a constant $c>0$ dependent only on $\left(\beta,\kappa_{0},\kappa_{1},\zeta,d\right)$ such that for any sequence $\left\{\vartheta_{i};i=1,\ldots,n\right\}$ of $\calF_{i-1}$-measurable $\Theta$-valued random variables $\vartheta_{i}$, $\tau\in\N_{0}$, and $n\in\N$,
    \begin{align*}
        \left|\E_{x}^{a,b}\left[\sum_{i=1}^{n-\tau}\left(F\left(\vartheta_{i};\xi_{i+\tau}\right)-H_{i+\tau,n}\left(\vartheta_{i}\right)\right)\right]\right|\le cnh_{n}^{\beta/2}\left(1+\sup_{t\ge0}\E_{x}^{a,b}\left[\left\|X_{t}^{a,b}\right\|_{2}^{c}\right]\right).
    \end{align*}
\end{proposition}

\begin{remark}
If $b$ is sufficiently smooth, the upper bound is $\mathcal{O}\left(nh_{n}\right)$ owing to It\^{o}'s formula \citep[see][]{Flo89,Kes97}.
\end{remark}

\begin{proof}
Lemma \ref{lem:st:NormEx} and the Markov property of $X_{t}$ lead to the existence of $c>0$ dependent only on $\left(\beta,\kappa_{0},\kappa_{1},\zeta,d\right)$ such that
\begin{align*}
    &\left|\E\left[\sum_{i=1}^{n-\tau}\left(F\left(\vartheta_{i};\xi_{i+\tau}\right)-H_{i+\tau,n}\left(\vartheta_{i}\right)\right)\right]\right|\\
    &=\left|\frac{1}{h_{n}}\sum_{i=1}^{n-\tau}\E\left[M_{i+\tau-1}\left[\bModel_{i+\tau-1}\left(\vartheta_{i}\right)-b_{i+\tau-1},\Delta_{i+\tau}X-h_{n}b_{i+\tau-1}\right]\right]\right|\\
    % &=\left|\frac{1}{h_{n}}\sum_{i=1}^{n-\tau}\E\left[\E\left[M_{i+\tau-1}\left[\bModel_{i+\tau-1}\left(\vartheta_{i}\right)-b_{i+\tau-1},\Delta_{i+\tau}X-h_{n}b_{i+\tau-1}\right]|\sigma\left(X_{t};t\le \left(i+\tau-1\right)h_{n}\right)\right]\right]\right|\\
    &=\left|\frac{1}{h_{n}}\sum_{i=1}^{n-\tau}\E\left[M_{i+\tau-1}\left[\bModel_{i+\tau-1}\left(\vartheta_{i}\right)-b_{i+\tau-1},\E\left[\Delta_{i+\tau}X-h_{n}b_{i+\tau-1}|X_{\left(i+\tau-1\right)h_{n}}\right]\right]\right]\right|\\
    &\le \frac{1}{h_{n}}\sum_{i=1}^{n-\tau}\E\left[\left\|M_{i+\tau-1}\right\|_{2}\left\|\bModel_{i+\tau-1}\left(\vartheta_{i}\right)-b_{i+\tau-1}\right\|_{2}\left\|\E\left[\Delta_{i+\tau}X-h_{n}b_{i+\tau-1}|X_{\left(i+\tau-1\right)h_{n}}\right]\right\|_{2}\right]\\
    &\le ch_{n}^{\beta/2}\sum_{i=1}^{n-\tau}\E\left[\left(1+\left\|X_{\left(i+\tau-1\right)h_{n}}\right\|_{2}^{c}\right)\right]\\
    &\le cnh_{n}^{\beta/2}\left(1+\sup_{t\ge0}\E\left[\left\|X_{t}\right\|_{2}^{c}\right]\right).
\end{align*}
Hence, the statement holds.% by Corollary \ref{cor:pr:Moments}.
\end{proof}

We obtain our main result on the drift estimation using learning rates whose optimality is attributable to \citet{DAJ+12}; we ignore the influence of $G$, $\check{K}$, and $R$.
\begin{theorem}\label{thm:st:b:UnifRiskBound}
    Assume that $h_{n}\in\left(0,1\right]$, $\log nh_{n}^{2}\ge 1$, and (D1)--(D2) hold.
    \begin{enumerate}
        \item[(i)] Under the update rule \eqref{SMDUpdate} with $\eta_{i}:=\eta h_{n}/\sqrt{i\log nh_{n}^{2}}$ and fixed $\eta>0$, for any $x\in\R^{d}$, there exists a positive constant $c>0$ dependent only on $\left(\varpi,\zeta,\eta,G,\check{K},R,d,x\right)$ such that
        \begin{align*}
            \sup_{\left(a,b\right)\in S_{\varpi}}\sup_{\theta\in\Theta}\E_{x}^{a,b}\left[\sum_{i=1}^{n}\left(f^{a,b}\left(\theta_{i}\right)-f^{a,b}\left(\theta\right)\right)\right]\le c\left(\sqrt{\frac{n}{h_{n}^{2}}}\sqrt{\log nh_{n}^{2}}+nh_{n}^{\beta/2}\right).
        \end{align*}
        \item[(ii)] Under the update rule \eqref{SMDUpdate} with $\eta_{i}:=\eta h_{n}/\sqrt{i}$ and fixed $\eta>0$, for any $x\in\R^{d}$, there exists a positive constant $c>0$ dependent only on $\left(\varpi,\zeta,\eta,G,\check{K},R,d,x\right)$ such that
        \begin{align*}
            \sup_{\left(a,b\right)\in S_{\varpi}}\sup_{\theta\in\Theta}\E_{x}^{a,b}\left[\sum_{i=1}^{n}\left(f^{a,b}\left(\theta_{i}\right)-f^{a,b}\left(\theta\right)\right)\right]\le c\left(\sqrt{\frac{n}{h_{n}^{2}}}\log nh_{n}^{2}+nh_{n}^{\beta/2}\right).
        \end{align*}
    \end{enumerate}
\end{theorem}

\begin{proof}
We present the proof only for (i) because that for (ii) is parallel.
The Hellinger distance and total variation distance have a property such that for an arbitrary pair of probability measures $P$ and $Q$ defined on the same measurable space,
\begin{align*}
    d_{\rm Hel}^{2}\left(P,Q\right)\le \left\|P-Q\right\|_{\TV}\le 2d_{\rm Hel}\left(P,Q\right).
\end{align*}
Corollary \ref{cor:pr:Moments} and Theorem \ref{thm:pr:Harris} yield that for any $\tau\ge1$, 
\begin{align*}
    \sup_{i\in\N_{0}}\E_{x}^{a,b}\left[d_{\rm Hel}^{2}\left(\left(P^{a,b}_{x}\right)_{\left[i\right]|\F}^{i+\tau},\Pi^{a,b}\right)\right]
    &\le \sup_{i\in\N_{0}}\E_{x}^{a,b}\left[\left.\left\|P_{\left(\tau-1\right)h_{n}}^{a,b}\left(x,\cdot\right)-\Pi^{a,b}\left(\cdot\right)\right\|_{\TV}\right|_{x=X_{ih_{n}}^{a,b}}\right]\\
    &\le \sup_{t\ge0}\E_{x}^{a,b}\left[c_{1}\exp\left(-\left(\tau-1\right)h_{n}/c_{2} \right)\left(V\left(X_{t}^{a,b}\right)+c_{3}\right)\right]\\
    &\le c_{1}'\exp\left(-\tau h_{n}/c_{2} \right)\left(\exp\left(\nu\sqrt{1+\left\|x\right\|_{2}^{2}}\right)+c_{3}'\right),
\end{align*}
where $c_{1}',c_{2},c_{3}'$, and $\nu$ are dependent only on $\left(\varpi,d\right)$.
Theorem \ref{thm:op:SMDEx}, Corollary \ref{cor:pr:Moments}, and Proposition \ref{prp:st:b:Resid} imply the existence of $c>0$ dependent only on $\left(\varpi,\zeta,G,\check{K},R,d,x\right)$ such that for all $\tau$ and $n$,
\begin{align*}
     \sup_{\theta\in\Theta}\E_{x}^{a,b}\left[\sum_{i=1}^{n}\left(f^{a,b}\left(\theta_{i}\right)-f^{a,b}\left(\theta\right)\right)\right]&\le cn\exp\left(-\tau h_{n}/c \right)+\frac{c\tau}{h_{n}}\sum_{i=1}^{n}\eta_{i}+\frac{c}{\eta_{n}} +c\tau +cnh_{n}^{\beta/2}.
\end{align*}
Letting $\tau=\left[c(\log nh_{n}^{2})/2h_{n}\right]$ gives
\begin{align*}
     \sup_{\theta\in\Theta}\E_{x}^{a,b}\left[\sum_{i=1}^{n}\left(f^{a,b}\left(\theta_{i}\right)-f^{a,b}\left(\theta\right)\right)\right]&\le c'\left(\sqrt{\frac{n}{h_{n}^{2}}}+\left(\eta+\eta^{-1}\right)\sqrt{\frac{n\log nh_{n}^{2}}{h_{n}^{2}}}+nh_{n}^{\beta/2}\right),
\end{align*}
where $c'>0$ is a constant dependent only on $\left(\varpi,\zeta,G,\check{K},R,d,x\right)$.
%$\sum_{i=1}^{n}1/\sqrt{i}\le 2\sqrt{n}$
\end{proof}

When we assume a usual identifiability condition \citep[see][]{UY11,Yos11,Yos22a,Yos22b} of the quasi-optimal parameter, that is, the optimal point $\theta^{a,b}_{0}$ of $f^{a,b}$, Theorem \ref{thm:st:b:UnifRiskBound} yields the rate of convergence.

\begin{corollary}\label{cor:st:b:Consist}
Assume that the same assumptions as in Theorem \ref{thm:st:b:UnifRiskBound} hold, the update rule \eqref{SMDUpdate} holds with $\eta_{i}:=\eta h_{n}/\sqrt{i}$ and $\eta>0$, and for all $\left(a,b\right)\in S_{\varpi}$, there exist $\chi^{a,b}>0$ and $\theta_{0}^{a,b}\in\Theta$ such that
\begin{align*}
    \frac{\chi^{a,b}}{2}\left\|\theta-\theta_{0}^{a,b}\right\|_{2}^{2}\le f^{a,b}\left(\theta\right)-f^{a,b}\left(\theta_{0}^{a,b}\right).
\end{align*}
\begin{enumerate}
    \item[(i)] There exists a positive constant $c>0$ dependent only on $\left(\varpi,\zeta,\eta,G,\check{K},R,d,x\right)$ such that
    \begin{align*}
        \sup_{\left(a,b\right)\in S_{\varpi}}\E_{x}^{a,b}\left[\frac{\chi^{a,b}}{2}\left\|\Bar{\theta}_{n}-\theta_{0}^{a,b}\right\|_{2}^{2}\right]&
        \le c\left(\frac{\log nh_{n}^{2}}{\sqrt{nh_{n}^{2}}}+h_{n}^{\beta/2}\right).
    \end{align*}
    \item[(ii)] If $nh_{n}^{2}\to\infty$ and $\sup_{n\in\N}nh_{n}^{2+\beta}<\infty$, then $\Bar{\theta}_{n}-\theta_{0}^{a,b}=\mathcal{O}_{P}\left(\sqrt[4]{\frac{nh_{n}^{2}}{(\log nh_{n}^{2})^{2}}}\right)$.
    \item[(iii)] If $nh_{n}^{2}\to\infty$ and $\sup_{n\in\N}nh_{n}^{2+\beta/4\rho}<\infty$ for some $\rho\in\left(0,1/4\right)$, then $\Bar{\theta}_{n}-\theta_{0}^{a,b}=\mathcal{O}_{P}\left(\left(nh_{n}^{2}\right)^{\rho}\right)$.
\end{enumerate}
\end{corollary}

% \begin{remark}
% The inequality in the assumption follows from the non-degeneracy of the information matrix .
% \end{remark}

\begin{proof}
(i) We obtain
\begin{align*}
    \frac{\chi^{a,b}}{2}\left\|\Bar{\theta}_{n}-\theta_{0}^{a,b}\right\|_{2}^{2}\le f\left(\Bar{\theta}_{n}\right)-f\left(\theta_{0}^{a,b}\right)\le \frac{1}{n}\sum_{i=1}^{n}\left(f\left(\theta_{i}\right)-f\left(\theta_{0}^{a,b}\right)\right).
\end{align*}
Hence, the first statement holds immediately.

(ii) The second statement holds because
\begin{align*}
    \sup_{\left(a,b\right)\in S_{\varpi}}\E_{x}^{a,b}\left[\frac{\chi^{a,b}}{2}\left\|\sqrt[4]{\frac{nh_{n}^{2}}{(\log nh_{n}^{2})^{2}}}\left(\Bar{\theta}_{n}-\theta_{0}^{a,b}\right)\right\|_{2}^{2}\right]\le c\left(1+\frac{\sqrt{nh_{n}^{2+\beta}}}{\log nh_{n}^{2}}\right).
\end{align*}

(iii) Similarly,
\begin{align*}
    \sup_{\left(a,b\right)\in S_{\varpi}}\E_{x}^{a,b}\left[\frac{\chi^{a,b}}{2}\left\|\left(nh_{n}^{2}\right)^{\rho}\left(\Bar{\theta}_{n}-\theta_{0}^{a,b}\right)\right\|_{2}^{2}\right]\le c\left(1+\left(nh_{n}^{2+\beta/4\rho}\right)^{2\rho}\right).
\end{align*}
This completes the proof.
\end{proof}

\subsubsection{Online gradient descent with the least-square loss function for drift coefficients}

We consider the least-square-type loss functions \citep[e.g., see][]{Mas05} for SDEs with drift coefficients linear in the parameters.
The target loss function is 
\begin{align}
    f^{a,b}\left(\theta\right)=\int \frac{1}{2}\left\|\bModel\left(\xi,\theta\right)-b\left(\xi\right)\right\|_{2}^{2}\Pi^{a,b}\left(\diff \xi\right).
\end{align}
It corresponds to the case $M\left(x\right)=I_{d}$ and $J\left(\theta\right)=0$ for all $x\in\R^{d}$ and $\theta\in N_{\Theta}$.
Hence, $H_{i,n}\left(\theta\right)$ is given as
\begin{align}
    H_{i,n}\left(\theta\right):=\frac{1}{2h_{n}^{2}}\left\|\Delta_{i}X-h_{n}\bModel\left(X_{\left(i-1\right)h_{n}},\theta\right)\right\|_{2}^{2}-\frac{1}{2h_{n}^{2}}\left\|\Delta_{i}X-h_{n}b\left(X_{\left(i-1\right)h_{n}}\right)\right\|_{2}^{2}.
\end{align}
We set the following assumption:
\begin{itemize}
    \item[(D2')] $\bModel\left(x,\theta\right)$ is in $\mathcal{C}^{1}\left(\R^{d}\times N_{\Theta}\right)$ and each component is linear in $\theta\in N_{\Theta}$ for all $x\in\R^{d}$, and there exists a constant $\zeta>0$ such that for all $x\in\R^{d}$ and $\theta\in N_{\Theta}$,
    \begin{align*}
        \left\|\bModel\left(x,\theta\right)\right\|_{2}\le \zeta\left(1+\left\|x\right\|_{2}^{\zeta}\right),\ 
        \left\|\partial_{\theta}\bModel\left(x,\theta\right)\right\|_{F}\le \zeta\left(1+\left\|x\right\|_{2}^{\zeta}\right).
    \end{align*}
\end{itemize}
Under (D2'), $H_{i,n}\left(\theta\right)$ is a.s.\ convex in $\theta$ and its gradient is given as
\begin{align}
    \sfK:=-\frac{1}{h_{n}}\left(\partial_{\theta}\bModel\right)\left(X_{\left(i-1\right)h_{n}},\theta\right)\left(\Delta_{i}X-h_{n}\bModel\left(X_{\left(i-1\right)h_{n}},\theta\right)\right).
\end{align}

By choosing $\eta_{i}=\eta h_{n}/\sqrt{i}$ with $\eta>0$, we obtain a simple update rule of the online gradient descent equivalent to \eqref{SMDUpdate}:
\begin{align}
    \theta_{i+1}:=\Proj_{\Theta}\left(\theta_{i}+\frac{\eta}{\sqrt{i}}\left(\partial_{\theta}\bModel\right)^{\top}\left(X_{\left(i-1\right)h_{n}},\theta_{i}\right)\left(\Delta_{i}X-h_{n}\bModel\left(X_{\left(i-1\right)h_{n}},\theta_{i}\right)\right)\right).\label{eq:OGDDriftUpdate}
\end{align}

\begin{lemma}\label{lem:st:b:ls:MomGrad}
    Assume that (D2') holds.
    There exists a constant $\check{K}>0$ dependent only on $\left(\varpi,\zeta,d,x\right)$ such that for all $i=1,\ldots,n$, $n\in\N$, $\calF_{i-1}$-measurable $\Theta$-valued random variables $\vartheta_{i}$,
    \begin{align*}
        \E_{x}^{a,b}\left[\left\|\frac{1}{h_{n}}\left(\partial_{\theta}\bModel_{i-1}\left(\vartheta_{i}\right)\right)^{\top}\left(\Delta_{i}X-h_{n}\bModel_{i-1}\left(\vartheta_{i}\right)\right)\right\|_{2}^{2}\right]\le \frac{\check{K}^{2}}{h_{n}}.
    \end{align*}
\end{lemma}

\begin{proof}
There exist constants $c,c'>0$ dependent only on $\left(\beta,\kappa_{0},\kappa_{1},\zeta,d\right)$ such that 
%There exists $c>0$ (which can vary from line to line) such that 
\begin{align*}
        &\E_{x}^{a,b}\left[\left\|\frac{1}{h_{n}}\left(\partial_{\theta}\bModel_{i-1}\left(\vartheta_{i}\right)\right)^{\top}\left(\Delta_{i}X-h_{n}\bModel_{i-1}\left(\vartheta_{i}\right)\right)\right\|_{2}^{2}\right]\\
        &\le \frac{1}{h_{n}^{2}}\E_{x}^{a,b}\left[\left\|\partial_{\theta}\bModel_{i-1}\left(\vartheta_{i}\right)\right\|_{F}^{2}\left\|\Delta_{i}X-h_{n}\bModel_{i-1}\left(\vartheta_{i}\right)\right\|_{2}^{2}\right]\\
        &\le \frac{2}{h_{n}^{2}}\E_{x}^{a,b}\left[\left\|\partial_{\theta}\bModel_{i-1}\left(\vartheta_{i}\right)\right\|_{F}^{2}\left\|\Delta_{i}X\right\|_{2}^{2}\right]+2\E_{x}^{a,b}\left[\left\|\partial_{\theta}\bModel_{i-1}\left(\vartheta_{i}\right)\right\|_{F}^{2}\left\|\bModel_{i-1}\left(\vartheta_{i}\right)\right\|_{2}^{2}\right]\\
        % &\le \frac{2}{h_{n}^{2}}\E_{x}^{a,b}\left[\left\|\partial_{\theta}\bModel_{i-1}\left(\vartheta_{i}\right)\right\|_{F}^{2}\E\left[\left\|\int_{\left(i-1\right)h_{n}}^{ih_{n}}b\left(X_{s}\right)\diff s+\int_{\left(i-1\right)h_{n}}^{ih_{n}}a\left(X_{s}\right)\diff w_{s}\right\|_{2}^{2}|X_{\left(i-1\right)h_{n}}\right]\right]\\
        % &\quad+\E_{x}^{a,b}\left[c\left(1+\left\|X_{\left(i-1\right)h_{n}}\right\|_{2}^{c}\right)\right]\\
        &\le \frac{1}{h_{n}}\E_{x}^{a,b}\left[c'\left(1+\left\|X_{\left(i-1\right)h_{n}}\right\|_{2}^{c'}\right)\right],
    \end{align*}
    by Lemma \ref{lem:st:ExNorm}.
    The proof is completed using this inequality and Corollary \ref{cor:pr:Moments}.% , and Lemma \ref{lem:st:ExNorm}.
\end{proof}

The existence of $G$ depending only on $\left(\varpi,\zeta,d,x\right)$ in (D1) is more obvious.
Hence, we obtain the following simple but useful corollary:

\begin{corollary}\label{cor:st:b:ls:OGD}
Under $h_{n}\in\left(0,1\right]$, $\log nh_{n}^{2}\ge 1$, the update rule \eqref{SMDUpdate} with $\eta_{i}=\eta h_{n}/\sqrt{i}$ and $\eta>0$, (D1) and (D2'), $\Bar{\theta}_{n}:=\frac{1}{n}\sum_{i=1}^{n}\theta_{i}$ has the uniform risk bound such that
    \begin{align*}
        \sup_{\left(a,b\right)\in S_{\varpi}}\sup_{\theta\in\Theta}\E_{x}^{a,b}\left[f^{a,b}\left(\Bar{\theta}_{n}\right)-f^{a,b}\left(\theta\right)\right]\le c\left(\frac{\log nh_{n}^{2}}{\sqrt{nh_{n}^{2}}}+h_{n}^{\beta/2}\right),
    \end{align*}
    where $c>0$ is a constant dependent only on $\left(\varpi,\zeta,\eta,R,d,x\right)$.
\end{corollary}

If the parameter space $\Theta$ is simple enough (e.g., rectangular, circular) that the projection on $\Theta$ is $\mathcal{O}\left(p\right)$, then the total computational complexity is $\mathcal{O}\left(np\right)$.
Such a complexity can be achieved by other gradient-based optimization algorithms such as the Stark--Parker algorithm; however, the estimation via online gradient descent achieves smaller computation costs in repetitive calling.
For any $m\in\N$ such that $m\le n$, the sequence of online estimators $\Bar{\theta}_{\left[n/m\right]},\ldots,\Bar{\theta}_{m\left[n/m\right]}$ can be composed with the complexity $\mathcal{O}\left(np\right)$ independent of $m$, whereas those of the other batch algorithms can be $\mathcal{O}\left(mnp\right)$.

\subsection{Estimation of diffusion coefficients}
We can estimate diffusion coefficients through arguments similar to drift ones.
\subsubsection{Statistical modelling}
Let $\Xi=\R^{d}$, $\Theta\in\mathcal{B}\left(\R^{p}\right)$ be the compact convex parameter space, $N_{\Theta}$ be an open neighbourhood of $\Theta$, and $R:=\sup\left\{\left\|\theta-\theta'\right\|_{2};\theta,\theta'\in\Theta\right\}$.
The discussion is much simpler than that of the drift estimation: we consider a known triple $(\AModel,M,J)$ of functions with Borel-measurable elements such that $\AModel:\R^{d}\times N_{\Theta}\to\R^{d}\otimes\R^{d}$ is the parametric model, $M:\R^{d}\to\R^{d}\otimes\R^{d}$ is the positive semi-definite matrix-valued weight function, and $J:N_{\Theta}\to\R$ is the regularization term.
Let $F\left(\theta;\xi\right)$ be a real-valued function defined as
\begin{align}
    F\left(\theta;\xi\right)=F^{a}\left(\theta;\xi\right):=\frac{1}{2}\left\|M^{1/2}\left(\xi\right)\left(A\left(\xi\right)-\AModel\left(\xi,\theta\right)\right)M^{1/2}\left(\xi\right)\right\|_{F}^{2}+J\left(\theta\right),
\end{align}
where $A\left(\xi\right):=a^{\otimes2}\left(\xi\right)$.
%and we assume its $\left(\mathcal{B}\left(\R^{p}\right)|_{N_{\Theta}}\right)\otimes\mathcal{B}\left(\R^{d}\right)$-measurability as well as the drift estimation.
Let us consider the minimization problem of $f^{a,b}$ defined as
\begin{align}
    f^{a,b}\left(\theta\right):=\int F^{a}\left(\theta;\xi\right)\Pi^{a,b}\left(\diff \xi\right)
\end{align}
on $\Theta$.
We provide the sequence of the sampled loss functions $F\left(\theta;\xi_{i}\right)$ with $\xi_{i}:=X_{\left(i-1\right)h_{n}}^{a,b}$ and its approximations $H_{i,n}$ such that
\begin{align}
    H_{i,n}\left(\theta\right)
    &:=\frac{1}{2}\left\|M_{i-1}^{1/2}\left(\frac{1}{h_{n}}\left(\Delta_{i}X\right)^{\otimes2}-\AModel_{i-1}\left(\theta\right)\right)M_{i-1}^{1/2}\right\|_{F}^{2}\notag\\
    &\quad-\frac{1}{2}\left\|M_{i-1}^{1/2}\left(\frac{1}{h_{n}}\left(\Delta_{i}X\right)^{\otimes2}-A_{i-1}\right)M_{i-1}^{1/2}\right\|_{F}^{2}+J\left(\theta\right)\notag\\
    % &=\frac{1}{2}\left\|M_{i-1}^{1/2}\left(\frac{1}{h_{n}}\left(\Delta_{i}X\right)^{\otimes2}-A_{i-1}+A_{i-1}-\AModel_{i-1}\left(\theta\right)\right)M_{i-1}^{1/2}\right\|_{F}^{2}\\
    % &\quad-\frac{1}{2}\left\|M_{i-1}^{1/2}\left(\frac{1}{h_{n}}\left(\Delta_{i}X\right)^{\otimes2}-A_{i-1}\right)M_{i-1}^{1/2}\right\|_{F}^{2}+J\left(\theta\right)\\
    &=\left(A_{i-1}-\AModel_{i-1}\left(\theta\right)\right)\left[M_{i-1}\left(\frac{1}{h_{n}}\left(\Delta_{i}X\right)^{\otimes2}-A_{i-1}\right)M_{i-1}\right]+F\left(\theta;\xi_{i}\right).
\end{align}

We present some assumptions.
\begin{itemize}
    \item[(V1)] For all $\xi,\xi'\in\R^{d}$, 
    \begin{align*}
        \frac{1}{2}\left\|M^{1/2}\left(\xi\right)\left(\left(\xi'\right)^{\otimes2}-\AModel\left(\xi,\theta\right)\right)M^{1/2}\left(\xi\right)\right\|_{F}^{2}+J\left(\theta\right)
    \end{align*}
    is convex with respect to $\theta\in N_{\Theta}$.
    Moreover, there exist positive constants $G>0$ and $K>0$, a $\mathcal{B}(\R^{d})\otimes\left(\mathcal{B}\left(\R^{p}\right)|_{N_{\Theta}}\right)$-measurable function $\sfG$, and an $\calF_{i}^{a,b}\otimes\left(\mathcal{B}\left(\R^{p}\right)|_{N_{\Theta}}\right)$-measurable random function $\sfK$ such that $\sfG\left(\theta;\xi\right)\in \partial F\left(\theta;\xi\right)$ and $\sfK_{i,n}\left(\theta\right)\in \partial H_{i,n}\left(\theta\right)$ a.s.\ for all $\xi\in\R^{d}$, $\theta\in \Theta$, and $i=1,\ldots,n$, and 
    % Moreover, there exist positive constants $G>0$ and $K>0$ such that for all $n\in\N$ and $\calF_{i-1}$-measurable $\Theta$-valued $\theta$,
    \begin{align*}
        \sup_{\left(a,b\right)\in S_{\varpi}}\sup_{i\in\N}\E_{x}^{a,b}\left[\left\|\sfG\left(\vartheta_{i};\xi_{i}\right)\right\|_{2}^{2}\right]\le G^{2},\ 
        \sup_{\left(a,b\right)\in S_{\varpi}}\sup_{i=1,\ldots,n}\E_{x}^{a,b}\left[\left\|\sfK_{i,n}\left(\vartheta_{i}\right)\right\|_{2}^{2}\right]\le K^{2}
    \end{align*}
    for all $n\in\N$ and sequence of $\calF_{i\wedge n-1}^{a,b}$-measurable $\Theta$-valued random variables $\vartheta_{i}$.
    % where $\sfG$ is a $\mathcal{B}(\R^{d})\otimes\left(\mathcal{B}\left(\R^{p}\right)|_{N_{\Theta}}\right)$-measurable function such that $\sfG\left(\theta;\xi_{i}\right)\in \partial F\left(\theta;\xi_{i}\right)$, and $\sfK$ is an $\calF_{i}\otimes\left(\mathcal{B}\left(\R^{p}\right)|_{N_{\Theta}}\right)$-measurable random function such that $\sfK_{i,n}\left(\theta\right)\in \partial H_{i,n}\left(\theta\right)$ a.s.\ for all $\theta\in\Theta$ (these subdifferentials are well-defined).
    \item[(V2)] There exists a constant $\zeta>0$ such that for all $x\in\R^{d}$ and $\theta\in N_{\Theta}$,
    \begin{align*}
        \left\|\AModel\left(x,\theta\right)\right\|_{F}\le \zeta\left(1+\left\|x\right\|_{2}^{\zeta}\right),\ 
        \left\|M\left(x\right)\right\|_{2}\le \zeta.
    \end{align*}
\end{itemize}

\subsubsection{Main results}
Firstly, we exhibit the bound for the residual terms $\mathcal{R}_{\tau,n}$.

\begin{proposition}\label{prp:st:a:Resid}
    Assume that (V2) holds.
    There exists a constant $c>0$ dependent only on $\left(\alpha,\kappa_{0},\kappa_{1},\zeta,d\right)$ such that for any sequence $\left\{\vartheta_{i};i=1,\ldots,n\right\}$ with $\calF_{i-1}$-measurable $\Theta$-valued random variables $\vartheta_{i}$, $\tau\in\N_{0}$, and $n\in\N$,
    \begin{align*}
        \left|\E_{x}^{a,b}\left[\sum_{i=1}^{n-\tau}\left(F\left(\vartheta_{i};\xi_{i+\tau}\right)-H_{i+\tau,n}\left(\vartheta_{i}\right)\right)\right]\right|\le cnh_{n}^{\alpha/2}\left(1+\sup_{t\ge0}\E_{x}^{a,b}\left[\left\|X_{t}^{a,b}\right\|_{2}^{c}\right]\right).
    \end{align*}
\end{proposition} 
\begin{proof}

By the It\^{o} isometry and noting the $\alpha$-H\"{o}lder continuity of $A=a^{\otimes2}$ obtained by ($H_{\alpha}^{a}$), Lemma \ref{lem:st:ExNorm} and a discussion quite similar to Lemma \ref{lem:st:NormEx} yield
\begin{align*}
    &\left\|\E\left[\frac{1}{h_{n}}\left(\Delta_{i}X\right)^{\otimes2}-A\left(X_{\left(i-1\right)h_{n}}\right)|X_{\left(i-1\right)h_{n}}\right]\right\|_{F}\\
    &=\left\|\E\left[\frac{1}{h_{n}}\left(\int_{\left(i-1\right)h_{n}}^{ih_{n}}b\left(X_{s}\right)\diff s+\int_{\left(i-1\right)h_{n}}^{ih_{n}}a\left(X_{s}\right)\diff w_{s}\right)^{\otimes2}-A\left(X_{\left(i-1\right)h_{n}}\right)|X_{\left(i-1\right)h_{n}}\right]\right\|_{F}\\
    &\le \left\|\E\left[\frac{1}{h_{n}}\left(\int_{\left(i-1\right)h_{n}}^{ih_{n}}b\left(X_{s}\right)\diff s\right)^{\otimes2}|X_{\left(i-1\right)h_{n}}\right]\right\|_{F}\\
    &\quad+\left\|\E\left[\frac{2}{h_{n}}\left(\int_{\left(i-1\right)h_{n}}^{ih_{n}}b\left(X_{s}\right)\diff s\right)\left(\int_{\left(i-1\right)h_{n}}^{ih_{n}}a\left(X_{s}\right)\diff w_{s}\right)^{\top}|X_{\left(i-1\right)h_{n}}\right]\right\|_{F}\\
    &\quad+\left\|\E\left[\frac{1}{h_{n}}\left(\int_{\left(i-1\right)h_{n}}^{ih_{n}}a\left(X_{s}\right)\diff w_{s}\right)^{\otimes2}-A\left(X_{\left(i-1\right)h_{n}}\right)|X_{\left(i-1\right)h_{n}}\right]\right\|_{F}\\
    &\le ch_{n}^{1/2}\left(1+\left\|X_{\left(i-1\right)h_{n}}\right\|_{2}\right)^{c}+\left\|\E\left[\frac{1}{h_{n}}\int_{\left(i-1\right)h_{n}}^{ih_{n}}A\left(X_{s}\right)\diff s-A\left(X_{\left(i-1\right)h_{n}}\right)|X_{\left(i-1\right)h_{n}}\right]\right\|_{F}\\
    &\le ch_{n}^{1/2}\left(1+\left\|X_{\left(i-1\right)h_{n}}\right\|_{2}\right)^{c}+\E\left[\sup_{s\in\left[\left(i-1\right)h_{n},ih_{n}\right]}\left\|A\left(X_{s}\right)-A\left(X_{\left(i-1\right)h_{n}}\right)\right\|_{F}|X_{\left(i-1\right)h_{n}}\right]\\
    &\le c'h_{n}^{\alpha/2}\left(1+\left\|X_{\left(i-1\right)h_{n}}\right\|_{2}\right)^{c'},
\end{align*}
where $c,c'$ are constants dependent only on $\left(\alpha,\kappa_{0},\kappa_{1},\zeta,d\right)$.
The rest of the proof is quite parallel to that of Proposition \ref{prp:st:b:Resid}.
\end{proof}

We obtain the theorem for the diffusion estimation.
%The almost same proof as Theorem \ref{thm:st:b:UnifRiskBound} works for the following result.
\begin{theorem}\label{thm:st:a:UnifRiskBound}
    Assume that $h_{n}\in\left(0,1\right]$, $\log nh_{n}\ge 1$, and (V1)--(V2) hold.
    \begin{enumerate}
        \item[(i)] Under the update rule \eqref{SMDUpdate} with $\eta_{i}=\eta\sqrt{h_{n}/i\log nh_{n}}$ and $\eta>0$, for any $x\in\R^{d}$, there exists a positive constant $c>0$ dependent only on $\left(\varpi,\zeta,\eta,G,K,R,d,x\right)$ such that
        \begin{align*}
            \sup_{\left(a,b\right)\in S_{\varpi}}\sup_{\theta\in\Theta}\E_{x}^{a,b}\left[\sum_{i=1}^{n}\left(f^{a,b}\left(\theta_{i}\right)-f^{a,b}\left(\theta\right)\right)\right]\le c\left(\sqrt{\frac{n}{h_{n}}}\sqrt{\log nh_{n}}+nh_{n}^{\alpha/2}\right).
        \end{align*}
        \item[(ii)] Under the update rule \eqref{SMDUpdate} with  $\eta_{i}=\sqrt{h_{n}/i}$ and $\eta>0$, for any $x\in\R^{d}$, there exists a positive constant $c>0$ dependent only on $\left(\varpi,\zeta,\eta,G,K,R,d,x\right)$ such that
        \begin{align*}
            \sup_{\left(a,b\right)\in S_{\varpi}}\sup_{\theta\in\Theta}\E_{x}^{a,b}\left[\sum_{i=1}^{n}\left(f^{a,b}\left(\theta_{i}\right)-f^{a,b}\left(\theta\right)\right)\right]\le c\left(\sqrt{\frac{n}{h_{n}}}\log nh_{n}+nh_{n}^{\alpha/2}\right).
        \end{align*}
    \end{enumerate}
    % where
    % \begin{align*}
    %     f^{a,b}\left(\theta\right):=\int F\left(\theta;\xi\right)\Pi^{a,b}\left(\diff \xi\right)=\int \left(M\left(\xi\right)\left[\left(\bModel\left(\xi,\theta\right)-b\left(\xi\right)\right)^{\otimes2}\right]+J\left(\theta\right)\right)\Pi^{a,b}\left(\diff \xi\right).
    % \end{align*}    
\end{theorem}

\begin{proof}
We only examine (i) because the proof of (ii) is analogous.
The argument that is parallel to that for Theorem \ref{thm:st:b:UnifRiskBound} leads to the existence of $c>0$ depending only on $\left(\varpi,\zeta,G,K,R,d,x\right)$ such that for all $\tau$ and $n$,
\begin{align*}
     \sup_{\theta\in\Theta}\E_{x}^{a,b}\left[\sum_{i=1}^{n}\left(f^{a,b}\left(\theta_{i}\right)-f^{a,b}\left(\theta\right)\right)\right]&\le cn\exp\left(-\tau h_{n}/c \right)+c\tau\sum_{i=1}^{n}\eta_{i}+\frac{c}{\eta_{n}} +c\tau +cnh_{n}^{\alpha/2}.
\end{align*}
Letting $\tau=\left[c(\log nh_{n})/2h_{n}\right]$ yields
\begin{align*}
     \sup_{\theta\in\Theta}\E_{x}^{a,b}\left[\sum_{i=1}^{n}\left(f^{a,b}\left(\theta_{i}\right)-f^{a,b}\left(\theta\right)\right)\right]&\le c'\left(\sqrt{\frac{n}{h_{n}}}+\left(\eta+\eta^{-1}\right)\sqrt{\frac{n\log nh_{n}}{h_{n}}}+nh_{n}^{\alpha/2}\right),
\end{align*}
where $c'>0$ is a constant dependent only on $\left(\varpi,\zeta,G,\check{K},R,d,x\right)$.
% because $\sum_{i=1}^{n}\eta_{i}\le 2\eta \sqrt{nh_{n}/\log nh_{n}}$.
This is the desired conclusion.
\end{proof}

\subsubsection{Online gradient descent with the least-square loss function for diffusion coefficients}
Let $M=I_{d}$ and $J\left(\theta\right)=0$, which define a least-square loss function.
Furthermore, we assume the following:
\begin{itemize}
    \item[(V2')] $\AModel\left(x,\theta\right)$ is in $\mathcal{C}^{1}\left(\R^{d}\times N_{\Theta}\right)$ and each component is linear in $\theta$ for all $x\in\R^{d}$, and there exists a constant $\zeta>0$ such that for all $x\in\R^{d}$ and $\theta\in N_{\Theta}$,
    \begin{align*}
        \left\|\AModel\left(x,\theta\right)\right\|_{F}\le \zeta\left(1+\left\|x\right\|_{2}^{\zeta}\right),\ \sqrt{\sum_{j=1}^{p}\left\|\partial_{\theta_{j}}\AModel\left(x,\theta\right)\right\|_{F}^{2}}\le \zeta\left(1+\left\|x\right\|_{2}^{\zeta}\right).
    \end{align*}
\end{itemize}

\begin{lemma}\label{lem:st:a:ls:MomGrad}
    Under (V2'), there exists a constant $K>0$ dependent only on $\left(\varpi,\zeta,d,x\right)$ such that for all $i=1,\ldots,n$, $n\in\N$, and $\calF_{i-1}$-measurable $\Theta$-valued random variables $\vartheta_{i}$,
    \begin{align*}
        \E_{x}^{a,b}\left[\left\|\partial_{\theta}\left(\AModel\left(X_{\left(i-1\right)h_{n}},\vartheta_{i}\right)\left[\frac{1}{h_{n}}\left(\Delta_{i}X\right)^{\otimes2}-A\left(X_{\left(i-1\right)h_{n}}\right)\right]\right)\right\|_{2}^{2}\right]\le K^{2}.
    \end{align*}
\end{lemma}
\begin{proof}
Corollary \ref{cor:pr:Moments} and Lemma \ref{lem:st:ExNorm} yield
\begin{align*}
    &\E\left[\sum_{j=1}^{p}\left(\partial_{\theta_{j}}\AModel\left(X_{\left(i-1\right)h_{n}},\vartheta_{i}\right)\left[\frac{1}{h_{n}}\left(\Delta_{i}X\right)^{\otimes2}-A\left(X_{\left(i-1\right)h_{n}}\right)\right]\right)^{2}\right]\\
    &\le \E\left[\left(\sum_{j=1}^{p}\left\|\partial_{\theta_{j}}\AModel\left(X_{\left(i-1\right)h_{n}},\vartheta_{i}\right)\right\|_{F}^{2}\right)\left\|\frac{1}{h_{n}}\left(\Delta_{i}X\right)^{\otimes2}-A\left(X_{\left(i-1\right)h_{n}}\right)\right\|_{F}^{2}\right]\\
    &\le c\E\left[\left(1+\left\|X_{\left(i-1\right)h_{n}}\right\|_{2}\right)^{c}\left(\frac{1}{h_{n}^{2}}\left\|\Delta_{i}X\right\|_{2}^{4}+\left\|A\left(X_{\left(i-1\right)h_{n}}\right)\right\|_{F}^{2}\right)\right]\\
    &\le K^{2}
\end{align*}
for some $c,K>0$ dependent only on $\left(\varpi,\zeta,d,x\right)$.
\end{proof}

The existence of $G$ dependent only on $\left(\varpi,\zeta,d,x\right)$ can be observed through an easier argument.
We obtain the following uniform risk bound:
\begin{corollary}\label{cor:st:a:ls:OGD}
Under $h_{n}\in\left(0,1\right]$, $\log nh_{n}\ge1$, the update rule \eqref{SMDUpdate} with $\eta_{i}:=\eta\sqrt{h_{n}/i}$ and $\eta>0$, (V1) and (V2'), $\Bar{\theta}_{n}:=\frac{1}{n}\sum_{i=1}^{n}\theta_{i}$ has the uniform risk bound such that
    \begin{align*}
        \sup_{\left(a,b\right)\in S_{\varpi}}\sup_{\theta\in\Theta}\E_{x}^{a,b}\left[f^{a,b}\left(\Bar{\theta}_{n}\right)-f^{a,b}\left(\theta\right)\right]\le c\left(\frac{\log nh_{n}}{\sqrt{nh_{n}}}+h_{n}^{\alpha/2}\right),
    \end{align*}
    where $c>0$ a positive constant dependent only on $\left(\varpi,\zeta,\eta,R,d,x\right)$
\end{corollary}
For the diffusion coefficient estimation, the upper bound is of order $\mathcal{O}\left(\frac{\log nh_{n}}{\sqrt{nh_{n}}}+h_{n}^{\alpha/2}\right)$, which is better than that of the drift estimation.

% \section{Numerical results}

\section{Discussion}
We obtain a uniform risk bound for the online parametric estimation of the drift and diffusion coefficients with misspecified modelling and discrete observations via online gradient descent.
Our estimator via online gradient descent is computationally efficient for repetitive construction, which is a favorable property for online decision-making, one of the most significant motivations for time series data analysis.
% The convergence rate is worse than in previous results; however, multi-step estimation is known to improve the convergence rate.
% Hence, our estimator can be used as an initial value for multi-step estimators with additional smoothness assumptions on the models \citep[for example,][]{KU15}.

Let us discuss some limitations of our results.
We obtain the present results under weak conditions on convexity and smoothness; however, the convergence rates are worse than those obtained by the batch estimation methods proposed by other studies.
The classes of SDEs are also restricted to Wiener-driven non-degenerate SDEs satisfying strong drift conditions.
Furthermore, the statistical problems considered in this study are restricted to the estimation of the drift and diffusion coefficients of such SDEs based on discrete observations without noises.
However, as the bound is an outcome of our contributions to SMD, the simultaneous ergodicity of SDEs, and the proposal of loss functions, modifying any of them would lead to new results solving those limitations.
We highlight some possible future research directions that can improve our results.

\subsection*{Stochastic mirror descent}
Our risk bound on SMD cannot be improved except for constant factors if we do not set additional assumptions \citep{DAJ+12}.
In i.i.d.\ cases, the stochastic gradient descent for strongly convex loss functions is known to achieve faster convergence.
Hence, if extending our result is possible in the case of strong convexity, we can obtain new online estimators with better convergence rates.
In addition, recent studies have shown the convergence of stochastic gradient descent to local minima for nonconvex loss functions.
Therefore, we can obtain the convergence of our estimators to local minima if our result on SMD is generalized in nonconvex cases.

\subsection*{Simultaneous ergodicity}
Our proof for the simultaneous ergodicity of classes of SDEs is heavily dependent on the Aronson-type estimates for transition densities and their gradients of SDEs under considerably weak conditions by \citet{MPZ21}.
Therefore, the same type of bounds with weaker assumptions can clearly broaden our class of coefficients $S_{\varpi}$.
In addition, similar bounds for the different classes of stochastic processes can provide similar simultaneous ergodicity and further extend our result.
%In additions, similar bounds for different classes of stochastic processes such as solutions of L\'{e}vy-driven SDEs can give similar simultaneous ergodicity and lead to extension of our result.
Note that the simultaneous sub-exponential ergodicity should hold by combining the discussions of \citet{Kul17} and \citet{MPZ21}.
This enables us to weaken the assumption if the deterioration of the risk bound is allowed.

\subsection*{Proposal of loss functions}
We use loss functions for drift and diffusion estimation and provide approximations whose gradients are independent of unknown functions/quantities.
\citet{Mas05} considers least-square-type estimation for a broad class of stochastic processes, including jump-diffusion processes; therefore, applying our discussion to model-wise risk bounds for the online parametric estimation of jump-diffusion processes is possible owing to the result of \citet{Mas07}.
Furthermore, the approximation of loss functions in our discussion is quite common in statistics of diffusion processes with partial observations such as integrated \citep{Glo00,Glo06} and noisy ones \citep{Fav14,Fav16}.
Hence, our approach applies even to the estimation of diffusion processes with partial observations.

\section{Conclusion}
The proposed estimation method for SDEs with discrete observations via online gradient descent has two desirable properties: computational efficiency and uniform risk guarantees even with model misspecification.
Its computational complexity is lower than existing batch estimation methods, and this improvement is particularly noticeable for online estimation.
In addition, the estimation has non-asymptotic risk bounds uniform in SDEs in certain classes, whose derivation has also been studied in the sequential estimation of SDEs.
Those bounds are obtained by SMD with dependent and biased subgradients and the simultaneous ergodicity of families of multidimensional SDEs.
These technical results are also new in gradient-based optimization algorithms for dependent data and the simultaneous ergodicity of classes of SDEs, respectively.
Notably, our study also provides theoretical guarantees for estimation with misspecified modelling.
Even when the models are misspecified, the estimators converge to the quasi-optimal parameters whose corresponding models of coefficients are close to the true coefficients in the $L^{2}$-distance with respect to the invariant probability measures.

\section*{Acknowledgements}
This work was partially supported by JSPS KAKENHI Grant Number JP21K20318.

\appendix

\section{Proofs of Theorem \ref{thm:op:SMDEx}}
% \subsection{Technical lemmas}
We present technical lemmas for the proof of Theorem \ref{thm:op:SMDEx}.
Note that some of them are indifferent to those by \citet{DAJ+12}.
\begin{lemma}[Lemma 6.1 of \citealp{DAJ+12}]\label{lem:op:DAJ+12L61}
    Let $\theta_{i}$ be defined by the SMD update \eqref{SMDUpdate}.
    For any $\tau\in\N_{0}$ and any $\theta^{\prime}\in\Theta$,
    \begin{align*}
        \sum_{i=\tau+1}^{n}\left(H_{i,n}\left(\theta_{i}\right)-H_{i,n}\left(\theta^{\prime}\right)\right)\le \frac{R^{2}}{2\eta_{n}}+\sum_{i=\tau+1}^{n}\frac{\eta_{i}\left\|\sfK_{i,n}\left(\theta_{i}\right)\right\|_{\ast}^{2}}{2}
    \end{align*}
    almost surely.
\end{lemma}

\begin{lemma}[Lemma 6.2 of \citealp{DAJ+12}]\label{lem:op:DAJ+12L62}
    Let $\theta_{i}$ be defined by the SMD update \eqref{SMDUpdate}.
    Then
    \begin{align*}
        \left\|\theta_{i}-\theta_{i+1}\right\|\le \eta_{i}\left\|\sfK_{i,n}\left(\theta_{i}\right)\right\|_{\ast}
    \end{align*}
    almost surely.
\end{lemma}

The following lemmas and propositions are extensions of those by \citet{DAJ+12}.

% [a version of Lemma 6.3 of \citealp{DAJ+12}]
\begin{lemma}\label{lem:op:DAJ+12L63}
    Let $i=0,\ldots,n$, $\vartheta_{i+1}$ be an $\calF_{i}$-measurable random variable, $\theta^{\prime}\in\Theta$ be an arbitrary non-random element, and $\tau\ge 1$.
    If Assumption (A1) holds,
    \begin{align*}
        &\left|\E\left[f\left(\vartheta_{i+1}\right)-f\left(\theta^{\prime}\right)-F\left(\vartheta_{i+1};\xi_{i+\tau}\right)+F\left(\theta^{\prime};\xi_{i+\tau}\right)|\calF_{i}\right]\right|\\
        &\le \sqrt{2}Rd_{\mathrm{Hel}}\left(P_{\left[i\right]}^{i+\tau},\Pi\right)\sqrt{2G^{2}+\E\left[\max\left\{\left\|\sfG\left(\vartheta_{i+1};\xi_{i+\tau}\right)\right\|_{\ast}^{2},\left\|\sfG\left(\theta^{\prime};\xi_{i+\tau}\right)\right\|_{\ast}^{2}\right\}|\calF_{i}\right]}
    \end{align*}
    almost surely.
\end{lemma}

\begin{proof}%[Proof of Lemma \ref{lem:op:DAJ+12L63}]
We have an integral representation such that
\begin{align*}
    &\E\left[f\left(\vartheta_{i+1}\right)-f\left(\theta^{\prime}\right)-F\left(\vartheta_{i+1};\xi_{i+\tau}\right)+F\left(\theta^{\prime};\xi_{i+\tau}\right)|\calF_{i}\right]\\
    &=\int\left(F\left(\vartheta_{i+1};\xi\right)-F\left(\theta^{\prime};\xi\right)\right)\Pi\left(\diff\xi\right)-\int\left(F\left(\vartheta_{i+1};\xi\right)-F\left(\theta^{\prime};\xi\right)\right)P_{\left[i\right]|\F}^{i+\tau}\left(\diff\xi\right).
\end{align*}
Let $p_{\left[s\right]|\F}^{t}$ and $\pi$ be the densities of $P_{\left[s\right]|\F}^{t}$ and $\Pi$ with respect to a reference measure $\mu$, respectively.
Then 
\begin{align*}
    &\int\left(F\left(\vartheta_{i+1};\xi\right)-F\left(\theta^{\prime};\xi\right)\right)\Pi\left(\diff\xi\right)-\int\left(F\left(\vartheta_{i+1};\xi\right)-F\left(\theta^{\prime};\xi\right)\right)P_{\left[i\right]|\F}^{i+\tau}\left(\diff\xi\right)\\
    &=\int\left(F\left(\vartheta_{i+1};\xi\right)-F\left(\theta^{\prime};\xi\right)\right)\left(\pi\left(\xi\right)-p_{\left[i\right]|\F}^{i+\tau}\left(\xi\right)\right)\mu\left(\diff\xi\right).
\end{align*}
As \citet{DAJ+12}, under Assumption (A1), it holds that
\begin{align*}
    &\left|\int\left(F\left(\vartheta_{i+1};\xi\right)-F\left(\theta^{\prime};\xi\right)\right)\left(\pi\left(\xi\right)-p_{\left[i\right]|\F}^{i+\tau}\left(\xi\right)\right)\mu\left(\diff\xi\right)\right|\\
    &\le d_{\mathrm{Hel}}\left(P_{\left[i\right]|\F}^{i+\tau},\Pi\right)\sqrt{\int\left(F\left(\vartheta_{i+1};\xi\right)-F\left(\theta^{\prime};\xi\right)\right)^{2}\left(\sqrt{\pi\left(\xi\right)}+\sqrt{p_{\left[i\right]|\F}^{i+\tau}\left(\xi\right)}\right)^{2}\mu\left(\diff\xi\right)}\\
    &\le \sqrt{2}d_{\mathrm{Hel}}\left(P_{\left[i\right]|\F}^{i+\tau},\Pi\right)\sqrt{\int\left(F\left(\vartheta_{i+1};\xi\right)-F\left(\theta^{\prime};\xi\right)\right)^{2}\left(\pi\left(\xi\right)+p_{\left[i\right]|\F}^{i+\tau}\left(\xi\right)\right)\mu\left(\diff\xi\right)}\\
    &=\sqrt{2}d_{\mathrm{Hel}}\left(P_{\left[i\right]|\F}^{i+\tau},\Pi\right)\\
    &\quad\times\sqrt{\left.\E_{\Pi}\left[\left(F\left(\theta;\xi\right)-F\left(\theta^{\prime};\xi\right)\right)^{2}\right]\right|_{\theta=\vartheta_{i+1}}+\E\left[\left(F\left(\vartheta_{i+1};\xi_{i+\tau}\right)-F\left(\theta^{\prime};\xi_{i+\tau}\right)\right)^{2}|\calF_{i}\right]}\\
    %&\le \sqrt{2}Rd_{\mathrm{Hel}}\left(P_{\left[i\right]|\F}^{i+\tau},\Pi\right)\\
    %&\quad\times\left(\left.\E_{\Pi}\left[\max\left\{\left\|\sfG\left(\theta;\xi\right)\right\|_{\ast}^{2}, \left\|\sfG\left(\theta^{\prime};\xi\right)\right\|_{\ast}^{2}\right\}\right]\right|_{\theta=\vartheta_{i+1}}\right.\\
    %&\quad\qquad \left.+\E\left[\max\left\{\left\|\sfG\left(\vartheta_{i+1};\xi_{i+\tau}\right)\right\|_{\ast}^{2},\left\|\sfG\left(\theta^{\prime};\xi_{i+\tau}\right)\right\|_{\ast}^{2}\right\}|\calF_{i}\right]\right)\\
    &\le \sqrt{2}Rd_{\mathrm{Hel}}\left(P_{\left[i\right]|\F}^{i+\tau},\Pi\right)\sqrt{2G^{2}+\E\left[\max\left\{\left\|\sfG\left(\vartheta_{i+1};\xi_{i+\tau}\right)\right\|_{\ast}^{2},\left\|\sfG\left(\theta^{\prime};\xi_{i+\tau}\right)\right\|_{\ast}^{2}\right\}|\calF_{i}\right]},
\end{align*}
where $\E_{\Pi}$ denotes the expectation with respect to $\xi\sim\Pi$, because 
\begin{align*}
    &\left.\E_{\Pi}\left[\left(F\left(\theta;\xi\right)-F\left(\theta^{\prime};\xi\right)\right)^{2}\right]\right|_{\theta=\vartheta_{i+1}}+\E\left[\left(F\left(\vartheta_{i+1};\xi_{i+\tau}\right)-F\left(\theta^{\prime};\xi_{i+\tau}\right)\right)^{2}|\calF_{i}\right]\\
    &\le R^{2}\left.\E_{\Pi}\left[\max\left\{\left\|\sfG\left(\theta;\xi\right)\right\|_{\ast}^{2}, \left\|\sfG\left(\theta^{\prime};\xi\right)\right\|_{\ast}^{2}\right\}\right]\right|_{\theta=\vartheta_{i+1}}\\
    &\quad+ R^{2}\E\left[\max\left\{\left\|\sfG\left(\vartheta_{i+1};\xi_{i+\tau}\right)\right\|_{\ast}^{2},\left\|\sfG\left(\theta^{\prime};\xi_{i+\tau}\right)\right\|_{\ast}^{2}\right\}|\calF_{i}\right]
\end{align*}
and for any $\theta\in\Theta$,
\begin{align*}
    \E_{\Pi}\left[\max\left\{\left\|\sfG\left(\theta;\xi\right)\right\|_{\ast}^{2},\left\|\sfG\left(\theta^{\prime};\xi\right)\right\|_{\ast}^{2}\right\}\right]\le \liminf_{n\to\infty}\E\left[\max\left\{\left\|\sfG\left(\theta;\xi_{n}\right)\right\|_{\ast}^{2},\left\|\sfG\left(\theta^{\prime};\xi_{n}\right)\right\|_{\ast}^{2}\right\}\right]\le 2G^{2}
\end{align*}
by Fatou's lemma and (A1).
This is the desired conclusion.
\end{proof}

%[a version of Lemma 6.4 by \citealp{DAJ+12}]
\begin{lemma}\label{lem:op:DAJ+12L64}
    Let $\theta_{i}$ be defined by the SMD update \eqref{SMDUpdate}, $\tau\in\N$, and $\eta_{i}$ be non-increasing.
    Then
    \begin{align*}
        &H_{i+\tau,n}\left(\theta_{i}\right)-H_{i+\tau,n}\left(\theta_{i+\tau}\right)\\
        &\le \eta_{i}\sum_{j=i}^{i+\tau-1}\max\left\{\left\|\sfK_{i+\tau,n}\left(\theta_{j}\right)\right\|_{\ast},\left\|\sfK_{i+\tau,n}\left(\theta_{j+1}\right)\right\|_{\ast}\right\}\left\|\sfK_{j,n}\left(\theta_{j}\right)\right\|_{\ast}
    \end{align*}
    almost surely.
\end{lemma}
\begin{proof}%[Proof of Lemma \ref{lem:op:DAJ+12L64}]
Recall that for any convex $f$ with a subgradient $g\left(x\right)$ at $x$, $f\left(x\right)-f\left(y\right)\le \langle g\left(x\right),x-y\rangle$ for all $y$.
We obtain
\begin{align*}
    &H_{i+\tau,n}\left(\theta_{j}\right)-H_{i+\tau,n}\left(\theta_{j+1}\right)\\
    &\le \left|H_{i+\tau,n}\left(\theta_{j}\right)-H_{i+\tau,n}\left(\theta_{j+1}\right)\right|\\
    &=\max\left\{H_{i+\tau,n}\left(\theta_{j}\right)-H_{i+\tau,n}\left(\theta_{j+1}\right),H_{i+\tau,n}\left(\theta_{j+1}\right)-H_{i+\tau,n}\left(\theta_{j}\right)\right\}\\
    &\le \max\left\{\langle \sfK_{i+\tau,n}\left(\theta_{j}\right),\theta_{j}-\theta_{j+1}\rangle,\langle \sfK_{i+\tau,n}\left(\theta_{j+1}\right),\theta_{j+1}-\theta_{j}\rangle\right\}\\
    &\le \max\left\{\left|\langle \sfK_{i+\tau,n}\left(\theta_{j}\right),\theta_{j}-\theta_{j+1}\rangle\right|,\left|\langle \sfK_{i+\tau,n}\left(\theta_{j+1}\right),\theta_{j+1}-\theta_{j}\rangle\right|\right\}\\
    &\le \max\left\{\left\|\sfK_{i+\tau,n}\left(\theta_{j}\right)\right\|_{\ast},\left\|\sfK_{i+\tau,n}\left(\theta_{j+1}\right)\right\|_{\ast}\right\}\left\|\theta_{j}-\theta_{j+1}\right\|.
\end{align*}
It then holds that
\begin{align*}
    &H_{i+\tau,n}\left(\theta_{i}\right)-H_{i+\tau,n}\left(\theta_{i+\tau}\right)\\
    &=\sum_{j=i}^{i+\tau-1}\left(H_{i+\tau,n}\left(\theta_{j}\right)-H_{i+\tau,n}\left(\theta_{j+1}\right)\right)\\
    &\le \sum_{j=i}^{i+\tau-1}\max\left\{\left\|\sfK_{i+\tau,n}\left(\theta_{j}\right)\right\|_{\ast},\left\|\sfK_{i+\tau,n}\left(\theta_{j+1}\right)\right\|_{\ast}\right\}\left\|\theta_{j}-\theta_{j+1}\right\|\\
    &\le \sum_{j=i}^{i+\tau-1}\max\left\{\left\|\sfK_{i+\tau,n}\left(\theta_{j}\right)\right\|_{\ast},\left\|\sfK_{i+\tau,n}\left(\theta_{j+1}\right)\right\|_{\ast}\right\}\eta_{j}\left\|\sfK_{j,n}\left(\theta_{j}\right)\right\|_{\ast}, 
\end{align*}
based on Lemma \ref{lem:op:DAJ+12L62}.
\end{proof}

\begin{proof}[Proof of Theorem \ref{thm:op:SMDEx}]
Note that $\theta_{i}$ is $\calF_{i-1}$-measurable by definition.
It holds that
\begin{align*}
    &\E\left[f\left(\theta_{i}\right)-f\left(\theta^{\prime}\right)-F\left(\theta_{i};\xi_{i+\tau-1}\right)+F\left(\theta^{\prime};\xi_{i+\tau-1}\right)\right]\\
    &\le \sqrt{2}R\E\left[d_{\mathrm{Hel}}^{2}\left(P_{\left[i-1\right]|\F}^{i+\tau-1},\Pi\right)\right]^{1/2}\E\left[2G^{2}+\E\left[\left\|\sfG\left(\theta_{i};\xi_{i+\tau-1}\right)\right\|_{\ast}^{2}+\left\|\sfG\left(\theta^{\prime};\xi_{i+\tau-1}\right)\right\|_{\ast}^{2}|\calF_{i-1}\right]\right]^{1/2}\\
    &\le  2\sqrt{2}GR\E\left[d_{\mathrm{Hel}}^{2}\left(P_{\left[i-1\right]|\F}^{i+\tau-1},\Pi\right)\right]^{1/2}
\end{align*}
by Lemma \ref{lem:op:DAJ+12L63}, the Cauchy--Schwarz inequality, and Assumption (A1).
Thus,
\begin{align*}
    &\sum_{i=1}^{n-\tau+1}\E\left[f\left(\theta_{i}\right)-f\left(\theta^{\prime}\right)-F\left(\theta_{i};\xi_{i+\tau-1}\right)+F\left(\theta^{\prime};\xi_{i+\tau-1}\right)\right]\\
    %&\le \sqrt{2}R\sum_{i=1}^{n-\tau+1}\sqrt{\E\left[d_{\mathrm{Hel}}^{2}\left(P_{\left[i-1\right]|\F}^{i+\tau-1},\Pi\right)\right]\left(G^{2}+\E\left[\max\left\{\left\|\sfG\left(\theta_{i};\xi_{i+\tau-1}\right)\right\|_{\ast}^{2},\left\|\sfG\left(\theta^{\prime};\xi_{i+\tau-1}\right)\right\|_{\ast}^{2}\right\}\right]\right)}\\
    &\le 2\sqrt{2}GR\sum_{i=1}^{n-\tau+1}\E\left[d_{\mathrm{Hel}}^{2}\left(P_{\left[i-1\right]|\F}^{i+\tau-1},\Pi\right)\right]^{1/2}.
\end{align*}
If $\tau=1$, $\sum_{i=1}^{n-\tau+1}\E\left[H_{i+\tau-1,n}\left(\theta_{i}\right)-H_{i+\tau-1,n}\left(\theta_{i+\tau-1}\right)\right]=0$ obviously.
Otherwise, Lemma \ref{lem:op:DAJ+12L64} and the Cauchy--Schwarz inequality lead to
\begin{align*}
    &\sum_{i=1}^{n-\tau+1}\E\left[H_{i+\tau-1,n}\left(\theta_{i}\right)-H_{i+\tau-1,n}\left(\theta_{i+\tau-1}\right)\right]\\
    &\le \sum_{i=1}^{n-\tau+1}\eta_{i}\sum_{j=i}^{i+\tau-2}\E\left[\max\left\{\left\|\sfK_{i+\tau,n}\left(\theta_{j}\right)\right\|_{\ast},\left\|\sfK_{i+\tau,n}\left(\theta_{j+1}\right)\right\|_{\ast}\right\}\left\|\sfK_{j,n}\left(\theta_{j}\right)\right\|_{\ast}\right]\\
    &\le \sum_{i=1}^{n-\tau+1}\eta_{i}\sum_{j=i}^{i+\tau-2}\E\left[\max\left\{\left\|\sfK_{i+\tau,n}\left(\theta_{j}\right)\right\|_{\ast}^{2},\left\|\sfK_{i+\tau,n}\left(\theta_{j+1}\right)\right\|_{\ast}^{2}\right\}\right]^{1/2}\E\left[\left\|\sfK_{j,n}\left(\theta_{j}\right)\right\|_{\ast}^{2}\right]^{1/2}\\
    &\le \sum_{i=1}^{n-\tau+1}\eta_{i}\sum_{j=i}^{i+\tau-2}\E\left[\left\|\sfK_{i+\tau,n}\left(\theta_{j}\right)\right\|_{\ast}^{2}+\left\|\sfK_{i+\tau,n}\left(\theta_{j+1}\right)\right\|_{\ast}^{2}\right]^{1/2}\E\left[\left\|\sfK_{j,n}\left(\theta_{j}\right)\right\|_{\ast}^{2}\right]^{1/2}\\
    &\le \sqrt{2}\left(\tau-1\right)K_{n}^{2}\sum_{i=1}^{n-\tau+1}\eta_{i},
\end{align*}
which holds even for $\tau=1$.
Lemma \ref{lem:op:DAJ+12L61} yields
\begin{align*}
    \sum_{i=\tau}^{n}\E\left[H_{i,n}\left(\theta_{i}\right)-H_{i,n}\left(\theta^{\prime}\right)\right]
    \le \frac{R^{2}}{2\eta_{n}}+\frac{1}{2}\sum_{i=\tau}^{n}\eta_{i}\E\left[\left\|\sfK_{i,n}\left(\theta_{i}\right)\right\|_{\ast}^{2}\right]
    \le \frac{R^{2}}{2\eta_{n}}+\frac{K_{n}^{2}}{2}\sum_{i=\tau}^{n}\eta_{i}.
\end{align*}
Naturally, 
\begin{align*}
    \E\left[f\left(\theta_{i}\right)-f\left(\theta^{\prime}\right)\right]\le \E\left[\left\langle \sfG\left(\theta_{i}\right),\theta_{i}-\theta^{\prime}\right\rangle\right]\le \E\left[\left\|\sfG\left(\theta_{i}\right)\right\|_{\ast}^{2}\right]^{1/2}\E\left[\left\|\theta_{i}-\theta^{\prime}\right\|^{2}\right]^{1/2}\le GR.
\end{align*}
Finally, we obtain
\begin{align*}
    \E\left[\sum_{i=1}^{n}\left(f\left(\theta_{i}\right)-f\left(\theta^{\prime}\right)\right)\right]
    &\le 2\sqrt{2}GR\sum_{i=1}^{n-\tau+1}\E\left[d_{\mathrm{Hel}}^{2}\left(P_{\left[i-1\right]}^{i+\tau-1},\Pi\right)\right]^{1/2}+\sqrt{2}\left(\tau-1\right)K_{n}^{2}\sum_{i=1}^{n-\tau+1}\eta_{i}\\
    &\quad+\frac{R^{2}}{2\eta_{n}}+\frac{K_{n}^{2}}{2}\sum_{i=\tau}^{n}\eta_{i}+\left(\tau-1\right)GR\\
    &\quad+\E\left[\resid_{\tau-1,n}\left(\left\{\theta_{i}\right\}\right)-\resid_{\tau-1,n}\left(\theta^{\prime}\right)\right]
\end{align*}
and the statement holds by substituting $\tau_{\E}\left(P_{|\F},\epsilon\right)$ for $\tau$.
\end{proof}

\section{Moments of diffusion processes}
We provide the following well-known results to check how the constants in the bounds are dependent on the parameters.
For simplicity, let $\E$ and $X_{t}$ denote $\E_{x}^{a,b}$ and $X_{t}^{a,b}$.

\begin{lemma}\label{lem:st:ExSup}
    Under ($H_{\alpha}^{a}$) and ($H_{\beta}^{b}$), for all $p>0$, there exists a constant $c>0$ dependent only on $\left(\kappa_{0},\kappa_{1},d,p\right)$ such that for all $t\ge0$,
    \begin{align*}
        \E\left[\sup_{u\in\left[t,t+1\right]}\left\|X_{u}\right\|_{2}^{p}|X_{t}\right]&\le c\left(1+\left\|X_{t}\right\|_{2}^{p}\right)
    \end{align*}
    almost surely.
\end{lemma}

\begin{proof}
Using H\"{o}lder's inequality, we assume that $p\ge1$ without loss of generality.
For any $u\in\left[t,t+1\right]$,
\begin{align*}
    f_{X_{t}}\left(u\right)&:=\E\left[\sup_{u^{\ast}\in\left[t,u\right]}\left\|X_{u^{\ast}}\right\|_{2}^{p}|X_{t}\right]\\
    &= \E\left[\sup_{u^{\ast}\in\left[t,u\right]}\left\|X_{t}+\int_{t}^{u^{\ast}}b\left(X_{u'}\right)\diff u'+\int_{t}^{u^{\ast}}a\left(X_{u'}\right)\diff w_{u'}\right\|_{2}^{p}|X_{t}\right]\\
    &\le 3^{p-1}\left\|X_{t}\right\|_{2}^{p}+3^{p-1}\left(\int_{t}^{u}\E\left[\kappa_{1}\left(2+\left\|X_{u'}\right\|_{2}\right)|X_{t}\right]\diff u'\right)^{p}\\
    &\quad+3^{p-1}c\E\left[\left(\int_{t}^{u}\left\|a\left(X_{u'}\right)\right\|_{F}^{2}\diff u'\right)^{p/2}|X_{t}\right]\\
    &\le 3^{p-1}\left\|X_{t}\right\|_{2}^{p}+6^{p-1}\kappa_{1}^{p}\int_{t}^{u}\E\left[2^{p}+\left\|X_{u'}\right\|_{2}^{p}|X_{t}\right]\diff u'+3^{p-1}c\left(\kappa_{0}d\right)^{p/2}\\
    &\le 3^{p-1}\left\|X_{t}\right\|_{2}^{p}+6^{p-1}\kappa_{1}^{p}\left(2^{p}+\int_{t}^{u}\E\left[\sup_{u^{\ast}\in\left[t,u'\right]}\left\|X_{u^{\ast}}\right\|_{2}^{p}|X_{t}\right]\diff u'\right)+3^{p-1}c\left(\kappa_{0}d\right)^{p/2}\\
    &=3^{p-1}\left(\left\|X_{t}\right\|_{2}^{p}+2^{2p-1}\kappa_{1}^{p}+c\left(\kappa_{0}d\right)^{p/2}\right)+6^{p-1}\kappa_{1}^{p}\int_{t}^{u}f_{X_{t}}\left(u'\right)\diff u',
\end{align*}
where $c$ is a constant dependent only on $p$ given by the Burkholder--Davis--Gundy inequality.
Hence, Gronwall's inequality yields the conclusion.
\end{proof}

\begin{lemma}\label{lem:st:ExNorm}
    Under ($H_{\alpha}^{a}$) and ($H_{\beta}^{b}$), for all $p>0$, there exists a constant $c>0$ dependent only on $\left(\kappa_{0},\kappa_{1},d,p\right)$ such that for all $t\ge0$ and $h\in\left(0,1\right]$
    \begin{align*}
        \E\left[\sup_{u\in\left[t,t+h\right]}\left\|X_{u}-X_{t}\right\|_{2}^{p}|X_{t}\right]\le ch^{p/2}\left(1+\left\|X_{t}\right\|_{2}^{p}\right)
    \end{align*}
    almost surely.
\end{lemma}

\begin{proof}
It is sufficient to consider $p\ge 1$ owing to H\"{o}lder's inequality. 
Clearly, for all $u\in\left[t,t+h\right]$,
\begin{align*}
    \left\|X_{u}-X_{t}\right\|_{2}^{p}\le 2^{p-1}\left\|\int_{t}^{u}b\left(X_{s}\right)\diff s\right\|_{2}^{p}+2^{p-1}\left\|\int_{t}^{u}a\left(X_{s}\right)\diff w_{s}\right\|_{2}^{p}.
\end{align*}
The Burkholder--Davis--Gundy inequality verifies the existence of $c>0$ dependent only on $p$ such that
\begin{align*}
    \E\left[\sup_{u\in\left[t,t+h\right]}\left\|\int_{t}^{u}a\left(X_{s}\right)\diff w_{s}\right\|_{2}^{p}|X_{t}\right]\le c\E\left[\left(\int_{t}^{t+h}\left\|a\left(X_{s}\right)\right\|_{F}^{2}\diff s\right)^{p/2}|X_{t}\right]\le c\left(\kappa_{0}dh\right)^{p/2}.
\end{align*}
Moreover, it holds that
\begin{align*}
    \sup_{u\in\left[t,t+h\right]}\left\|\int_{t}^{u}b\left(X_{s}\right)\diff s\right\|_{2}^{p}&\le \left(\int_{t}^{t+h}\left\|b\left(X_{s}\right)\right\|_{2}\diff s\right)^{p}\le h^{p}\kappa_{1}^{p}\left(2+\sup_{s\in\left[t,t+h\right]}\left\|X_{s}\right\|_{2}\right)^{p}.
\end{align*}
We therefore obtain the result using Lemma \ref{lem:st:ExSup}.
\end{proof}

\begin{lemma}\label{lem:st:NormEx}
    Under ($H_{\alpha}^{a}$) and ($H_{\beta}^{b}$), for all $\beta\in\left[0,1\right]$, there exists a constant $c>0$ dependent only on $\left(\beta,\kappa_{0},\kappa_{1},d\right)$ such that for all $t\ge0$ and $h\in\left(0,1\right]$,
    \begin{align*}
        \left\|\E\left[X_{t+h}-X_{t}-hb\left(X_{t}\right)|X_{t}\right]\right\|_{2}&\le c\left(1+\left\|X_{t}\right\|_{2}\right)\left(h^{1+\beta/2}\right)
    \end{align*}
    almost surely.
\end{lemma}

\begin{proof}
Clearly, for any $t\ge0$ and $h\in\left(0,1\right]$
\begin{align*}
    \left\|\E\left[X_{t+h}-X_{t}-hb\left(X_{t}\right)|X_{t}\right]\right\|_{2}&=\left\|\E\left[\int_{t}^{t+h}\left(b\left(X_{u}\right)-b\left(X_{t}\right)\right)\diff u|X_{t}\right]\right\|_{2}\\
    &\le \E\left[\int_{t}^{t+h}\left\|b\left(X_{u}\right)-b\left(X_{t}\right)\right\|_{2}\diff u|X_{t}\right]\\
    &\le \int_{t}^{t+h}\kappa_{1}\E\left[\left\|X_{u}-X_{t}\right\|_{2}^{\beta}+\left\|X_{u}-X_{t}\right\|_{2}|X_{t}\right]\diff u.
\end{align*}
% The It\^{o} isometry yields that for all $\beta\in\left(0,1\right]$,
% \begin{align*}
%     \E\left[\left\|X_{s}-X_{t}\right\|_{2}^{\beta}|X_{t}\right]&\le \E\left[\left\|\int_{t}^{s}b\left(X_{u}\right)\diff u\right\|_{2}^{\beta}+\left\|\int_{t}^{s}a\left(X_{u}\right)\diff w_{u}\right\|_{2}^{\beta}|X_{t}\right]\\
%     &\le \E\left[\int_{t}^{s}\left\|b\left(X_{u}\right)\right\|_{2}\diff u|X_{t}\right]^{\beta}+\E\left[\left\|\int_{t}^{s}a\left(X_{u}\right)\diff w_{u}\right\|_{2}^{2}|X_{t}\right]^{\beta/2}\\
%     &\le \E\left[\int_{t}^{s}\kappa_{1}\left(2+\left\|X_{u}\right\|_{2}\right)\diff u|X_{t}\right]^{\beta}+\E\left[\int_{t}^{s}\left\|a\left(X_{u}\right)\right\|_{F}^{2}\diff u|X_{t}\right]^{\beta/2}\\
%     &\le \left(\kappa_{1}\left(s-t\right)\right)^{\beta}\E\left[2+\sup_{u\in\left[t,t+h\right]}\left\|X_{u}\right\|_{2}|X_{t}\right]^{\beta}+\left(\kappa_{0}d(s-t)\right)^{\beta/2}.
% \end{align*}
% As evident,
% \begin{align*}
%     \E\left[\sup_{u\in\left[t,t+h\right]}\left\|X_{u}\right\|_{2}|X_{t}\right]\le \left\|X_{t}\right\|_{2}+ \kappa_{0}\int_{t}^{t+h}\left(2+\E\left[\sup_{u\in\left[t,s\right]}\left\|X_{u}\right\|_{2}|X_{t}\right]\right)\diff s+\left(\kappa_{0}d h\right)^{1/2}
% \end{align*}
% and Gronwall's inequality verifies that
% \begin{align*}
%     \E\left[\sup_{u\in\left[t,t+h\right]}\left\|X_{u}\right\|_{2}|X_{t}\right]\le e^{\kappa_{0}h}\left(\left\|X_{t}\right\|_{2}+2\kappa_{0}h+\left(\kappa_{0}dh\right)^{1/2}\right).
% \end{align*}
Lemma \ref{lem:st:ExNorm} yields that for all $\beta\in\left(0,1\right]$, there exists a constant $c>0$ dependent only on $\left(\beta,\kappa_{0},\kappa_{1},d\right)$ such that
\begin{align*}
    \left\|\E\left[X_{t+h}-X_{t}-hb\left(X_{t}\right)|X_{t}\right]\right\|_{2}&\le c \left(1+\left\|X_{t}\right\|_{2}\right)\int_{0}^{h}\left(u^{\beta/2}+u^{1/2}\right)\diff u\\
    &\le c\left(1+\left\|X_{t}\right\|_{2}\right)\left(h^{1+\beta/2}\right).
\end{align*}
Naturally, the same bound holds for $\beta=0$.
\end{proof}

\bibliographystyle{apalike-month}
\bibliography{bibOPE}

\end{document}